\documentclass[11pt,reqno]{amsart}

\usepackage[margin=1.15in]{geometry}
\usepackage{amsmath,amssymb,amsthm,mathtools}
\usepackage[T1]{fontenc}
\usepackage{lmodern}
\usepackage[expansion=false]{microtype}
\usepackage{xcolor}
\usepackage[colorlinks=true,linkcolor=blue!55!black,citecolor=blue!55!black,
            urlcolor=blue!55!black]{hyperref}
\hypersetup{
  pdftitle={Kahler-Einstein cone metrics and parabolic bundles},
  pdfauthor={Martin de Borbon},
  pdfsubject={Complex differential geometry},
  pdfkeywords={Kahler-Einstein metrics; cone singularities; parabolic bundles;
               Hermitian-Einstein metrics; Miyaoka-Yau inequality}}

\newtheorem{theorem}{Theorem}[section]
\newtheorem{proposition}[theorem]{Proposition}
\newtheorem{lemma}[theorem]{Lemma}
\newtheorem*{theoremA}{Theorem A}
\newtheorem*{theoremB}{Theorem B}
\newtheorem*{theoremC}{Theorem C}
\newtheorem*{theoremD}{Theorem D}
\theoremstyle{definition}
\newtheorem{definition}[theorem]{Definition}

\theoremstyle{remark}
\newtheorem{remark}[theorem]{Remark}

\newcommand{\C}{\mathbb{C}}
\newcommand{\R}{\mathbb{R}}
\newcommand{\Xo}{X^{\circ}}
\newcommand{\cO}{\mathcal{O}}
\newcommand{\cV}{\mathcal{V}}
\newcommand{\rk}{\operatorname{rk}}
\newcommand{\im}{\operatorname{im}}
\newcommand{\Gr}{\operatorname{Gr}}
\newcommand{\Ric}{\operatorname{Ric}}
\newcommand{\tr}{\operatorname{tr}}
\newcommand{\End}{\operatorname{End}}
\newcommand{\Hom}{\operatorname{Hom}}
\newcommand{\Id}{\operatorname{Id}}
\newcommand{\Vol}{\operatorname{Vol}}
\newcommand{\Res}{\operatorname{Res}}
\newcommand{\ch}{\operatorname{ch}}
\newcommand{\pardeg}{\operatorname{par\,deg}}
\newcommand{\parc}{\operatorname{par\,c}}
\newcommand{\parch}{\operatorname{par\,ch}}
\newcommand{\supp}{\operatorname{supp}}
\newcommand{\Tlog}{T_X(-\log D)}
\newcommand{\OmLog}{\Omega_X^1(\log D)}

\begin{document}

\title[KE metrics and parabolic bundles]
      {K\"ahler--Einstein cone metrics and parabolic bundles}

\author{Martin de Borbon}
\address{Loughborough University, Loughborough LE11 3TU, UK}
\email{m.de-borbon@lboro.ac.uk}
\date{}
\subjclass[2020]{Primary 32Q20, 53C55; Secondary 14J60, 32W20}
\keywords{K\"ahler--Einstein metrics, cone singularities, parabolic bundles, Hermitian--Einstein metrics, Miyaoka--Yau inequality}

\begin{abstract}
Let \(X\) be a compact K\"ahler manifold, let \(D=\sum_iD_i\) be a simple normal crossing
divisor, and fix \(\beta_i\in(0,1)\). We introduce a parabolic structure \(TX_*\) on
\(TX\), given along each \(D_i\) by the filtration \(0\subset TD_i\subset TX|_{D_i}\) with
weight \(w_i=1-\beta_i\) on the normal quotient. If \(X\setminus D\) carries a
K\"ahler--Einstein metric \(\omega\) with cone angle \(2\pi\beta_i\) along \(D_i\), we show
that \(TX_*\) is parabolic polystable with respect to \([\omega]\). When the Einstein
constant \(\lambda\) is non-zero, we introduce parabolic structures on the Tian extension
(\(\lambda>0\)) and on Simpson's canonical Higgs bundle (\(\lambda<0\)), and prove that
these are polystable too. Combined with the Bogomolov--Gieseker inequality for parabolic
bundles, this yields a Miyaoka--Yau inequality for these pairs.
\end{abstract}

\maketitle

\section{Introduction}

The tangent bundle of a compact K\"ahler--Einstein manifold \((X^n, \omega)\) is slope polystable: the
K\"ahler--Einstein metric is Hermitian--Einstein on \(TX\), and polystability follows from
the easy direction of the Kobayashi--Hitchin correspondence \cite[Ch.~V, \S8]{Kobayashi}.
When the Einstein constant is positive, Tian \cite{Tian92} proved the finer statement that
the extension of \(TX\) by \(\cO_X\) defined by the K\"ahler class \([\omega]\) is polystable; when it
is negative, Simpson \cite{Simpson} defined a canonical Higgs bundle on \(\cO_X\oplus TX\)
which is polystable. Together with the Bogomolov--Gieseker inequality, the two constructions recover the
Miyaoka--Yau inequality
\[
   \bigl(2(n+1)c_2(X)-n\,c_1(X)^2\bigr)\cdot[\omega]^{n-2}\ge0 .
\]
In this paper we develop a parabolic analogue of this
picture for K\"ahler--Einstein metrics with cone singularities along a simple normal
crossing divisor.

Let \(X\) be a compact K\"ahler manifold of complex dimension \(n\) and let
\(D=\sum_{i=1}^{N}D_i\) be a simple normal crossing divisor with irreducible
components \(D_i\). Fix cone angles \(2\pi\beta_i\) with \(0<\beta_i<1\) and put
\[
   w_i:=1-\beta_i,
   \qquad
   \Delta:=\sum_{i=1}^{N}w_iD_i,
   \qquad
   \Xo:=X\setminus D.
\]
Our main results are as follows.

\begin{theoremA}
Let \(\omega\) be a K\"ahler--Einstein metric on \(\Xo\) with cone angle \(2\pi\beta_i\)
along \(D_i\). Endow \(TX\) with the parabolic structure \(TX_*\) of
Definition~\ref{def:parTX}: along \(D_i\), it is determined by the subbundle \(TD_i\subset TX|_{D_i}\) and the weight
\(w_i\). Then \(TX_*\) is parabolic
polystable with respect to \(\Omega:=[\omega]\).
\end{theoremA}

\begin{theoremB}
Let \(\omega\) be as in Theorem~A and suppose that \(\Ric(\omega)=\lambda\omega\) with
\(\lambda>0\). Let \(E\) be the extension \(0\to\cO_X\to E\xrightarrow{p}TX\to0\) whose
class is a non-zero multiple of \(\Omega\). Endow \(E\) with the parabolic structure
\(E_*\) of Definition~\ref{def:Tian-parabolic}: along \(D_i\), it is determined by the subbundle
\(p^{-1}(TD_i)\subset E|_{D_i}\) and the weight \(w_i\).
Then \(E_*\) is parabolic polystable with respect to \(\Omega\).
\end{theoremB}

\begin{theoremC}
Let \(\omega\) be as in Theorem~A and suppose that \(\Ric(\omega)=\lambda\omega\) with
\(\lambda<0\). Endow \(E:=\cO_X\oplus TX\) with the parabolic Higgs bundle structure
\((E_*,\theta)\) of Definition~\ref{def:HSbundle}: along \(D_i\), it is determined by the subbundle 
\(\cO_{D_i}\oplus TD_i\subset E|_{D_i}\) and the weight \(w_i\); and \(\theta\) is the
tautological Higgs field.
Then \((E_*,\theta)\) is parabolic Higgs polystable with respect to \(\Omega\).
\end{theoremC}

Write \(c_i(X):=c_i(TX)\),
\([D_i]:=c_1(\cO_X(D_i))\), and \(\iota_i:D_i\hookrightarrow X\) for the inclusion.
Following Tian \cite{TianCIME}, set
\begin{align}
   c_1(X,\Delta)&:=c_1(X)-\sum_iw_i[D_i],\label{eq:logc1}\\
   c_2(X,\Delta)&:=c_2(X)-\sum_iw_i\,(\iota_i)_*c_1(TD_i)
   +\sum_{i<j}w_iw_j[D_i]\cdot[D_j].\label{eq:logc2}
\end{align}
Combining Theorems A, B, C with the Bogomolov--Gieseker inequality for parabolic bundles
gives the following.

\begin{theoremD}
Assume \(n\ge2\) and let \(\omega\) be as in Theorem~A.
Then
\begin{equation}\label{eq:MYintro}
   \Bigl(2(n+1)c_2(X,\Delta)-n\,c_1(X,\Delta)^2\Bigr)\cdot\Omega^{n-2}\ge0 .
\end{equation}
\end{theoremD}

\subsection*{Previous work}
Precedents for this line of work are Guenancia \cite{Guenancia2016} and Guenancia--Taji
\cite{GuenanciaTaji}, on the stability of logarithmic and orbifold tangent sheaves and the
resulting Miyaoka--Yau inequality. Closer to Theorem~B is Li \cite[Theorem~1.4]{ChiLi},
on the stability of the extension of the logarithmic tangent sheaf on a log-Fano
pair. The boundary divisors in \cite{GuenanciaTaji, ChiLi} carry rational coefficients, while here \(\beta_i\in(0,1)\) is an arbitrary real number. 
The closest precedent to Theorem~A is Keller--Zheng \cite[Theorem~6.2]{KellerZheng}: they treat a smooth divisor under a condition
requiring \(\beta<\tfrac12\), and leave the simple normal crossing case open
\cite[\S7.1]{KellerZheng}.

For a smooth divisor, Song--Wang \cite{SongWang} prove \eqref{eq:MYintro} from a direct
Chern--Weil formula for conical K\"ahler--Einstein metrics using
the polyhomogeneous expansion of Jeffres--Mazzeo--Rubinstein \cite{JMR}. The route taken here is different:
Theorem~D is deduced from the parabolic Bogomolov--Gieseker inequality and polystability.

\subsection*{Outline}
Section~\ref{sec:parabolic} collects the background: parabolic bundles, a particular class of Hermitian metrics adapted to a parabolic structure that we call \emph{uniformly adapted metrics}, the Chern--Weil formula for the first Chern class of a saturated subsheaf, and K\"ahler cone metrics. 
Section~\ref{sec:technical} proves a criterion for parabolic polystability of a uniformly adapted Hermitian--Einstein metric over a K\"ahler cone metric on the base \(\Xo\), Theorem~\ref{thm:HEcriterion}. 
Its main analytic input is Theorem~\ref{thm:degree}, a Chern--Weil formula identifying the analytic degree of a saturated subsheaf with its parabolic degree. 
The criterion is applied to the parabolic tangent bundle and to the Tian extension in Sections~\ref{sec:thmA} and~\ref{sec:thmB}, which prove Theorems~A and~B. 
Section~\ref{sec:thmC} defines a parabolic structure on Simpson's canonical Higgs bundle
and proves Theorem~C. Section~\ref{sec:thmD} computes the parabolic Chern classes and applies the
parabolic Bogomolov--Gieseker inequality of Jiang--Li \cite{JiangLiBG} to prove Theorem~D.

\subsection*{AI declaration}
ChatGPT was used during the preparation of this manuscript for language editing, checking exposition, and discussing and checking aspects of the mathematical arguments. I take full responsibility for the mathematical content of the paper.

\section{Background}\label{sec:parabolic}

\subsection{Parabolic bundles}
We follow Mochizuki \cite[Ch.~3]{Mochizuki} and define parabolic bundles using increasing filtrations with weights in \([0,1)\); see also \cite[\S4.1]{dBP}. 
A \emph{parabolic bundle} on \((X, D)\), denoted by \(E_*\), consists of a holomorphic vector bundle
\(E\) on \(X\), called its underlying bundle, together with an increasing filtration by
holomorphic subbundles
\[
   F^i_c(E_*)\subset E|_{D_i},\qquad c\in[0,1),
\]
for every irreducible component \(D_i\) of \(D\).
The filtration has finitely many jumps, is right continuous
(\(F^i_c=\bigcap_{c'>c}F^i_{c'}\)), and exhaustive in the sense that
\(F^i_c=E|_{D_i}\) for \(c\) sufficiently close to \(1\). We set
\(F^i_{<0}:=0\), write \(F^i_{<c}:=\bigcup_{0\le c'<c}F^i_{c'}\) for \(c>0\), and put
\[
\Gr^i_c(E_*):=F^i_c/F^i_{<c} .
\]
The \emph{weights} along \(D_i\) are the jump values:
\[
   \operatorname{Wt}_i(E_*):=
   \bigl\{c\in[0,1):\Gr^i_c(E_*)\ne0\bigr\}.
\]

Near a point where \(D\) has \(\ell\) branches, relabel the components through the point
and choose \emph{adapted coordinates} \((z_1,\dots,z_n)\) on a neighbourhood \(U\) of the
point such that
\begin{equation}\label{eq:adap-coor}
   D\cap U=\{z_1\cdots z_\ell=0\},
   \qquad D_j\cap U=\{z_j=0\},\quad 1\le j\le\ell .
\end{equation}
We call \(U\) an \emph{adapted chart}. This relabelling is used in every local
computation: in an adapted chart, the index \(j\) always ranges over the components
\(D_1,\dots,D_\ell\) meeting the chart, and objects attached to \(D_i\), such as the
functions \(\rho_i\) of Section~\ref{ss:adapted} or the numbers \(\delta_i(\cV)\) of
\eqref{eq:deltai}, are written with the local index \(j\).

\begin{definition}\label{def:locallyabelian}
A parabolic bundle \(E_*\) is \emph{locally abelian} if every point of \(D\) has adapted
coordinates \eqref{eq:adap-coor} carrying a holomorphic frame \(e_1,\dots,e_r\) of \(E\)
such that
\begin{equation}\label{eq:adaptedframe}
   F^j_c(E_*)
   =\operatorname{span}_{\cO_{D_j}}
   \bigl\{e_\alpha|_{D_j}:
   \operatorname{wt}_{D_j}(e_\alpha)\le c\bigr\}
\end{equation}
for all \(1\le j\le\ell\) and \(c\in[0,1)\), where
\begin{equation}\label{eq:frameweight}
   \operatorname{wt}_{D_j}(e_\alpha)
   :=\min\bigl\{c\in[0,1):
        e_\alpha|_{D_j}\in F^j_c(E_*)\bigr\} \,.
\end{equation}
Such a frame is called \emph{adapted}.
\end{definition}

\begin{remark}
Definition~\ref{def:locallyabelian} is Mochizuki's compatibility condition
\cite[Definition~3.12 and Remark~3.13]{Mochizuki}; we use the terminology
\emph{locally abelian} of Iyer--Simpson \cite[\S2]{IyerSimpson}, see also \cite[\S4.1]{dBP}.
\end{remark}

For an arbitrary filtration multi-index
\(\mathbf c=(c_1,\dots,c_N)\in[0,1)^N\), define the \emph{constituent sheaf}
\begin{equation}\label{eq:constituentglobal}
   E_{\mathbf c}
   :=\ker\Bigl(E\longrightarrow
      \bigoplus_{i=1}^N\bigl(E|_{D_i}\bigr)/F^i_{c_i}(E_*)\Bigr).
\end{equation}
If \(E_*\) is locally abelian, then in adapted coordinates \eqref{eq:adap-coor} with an
adapted frame,
\begin{equation}\label{eq:constituents}
   E_{\mathbf c}=\bigoplus_{\alpha=1}^r\cO_U\cdot m_\alpha(\mathbf c) \, e_\alpha,
   \qquad
   m_\alpha(\mathbf c)
   :=\prod_{\{1\le j\le\ell\,:\,
      c_j<\operatorname{wt}_{D_j}(e_\alpha)\}}z_j ,
\end{equation}
so \(E_{\mathbf c}\) is locally free of rank \(r\). The constituent sheaf corresponding to the
zero multi-index is denoted by \(E_{\mathbf 0}\).

\begin{definition}\label{def:pardeg}
The parabolic first Chern class of a parabolic bundle \(E_*\) is
\begin{equation}\label{eq:parc1}
   \parc_1(E_*):=c_1(E)-\sum_{i=1}^N\sum_{c\in(0,1)}c\,\rk\Gr^i_c(E_*)\,[D_i]
   \ \in H^2(X,\R).
\end{equation}
For a K\"ahler class \(\Omega\),
\begin{equation}\label{eq:pardeg}
   \pardeg_\Omega(E_*):=\int_X\parc_1(E_*)\wedge\frac{\Omega^{n-1}}{(n-1)!},
   \qquad
   \mu_\Omega(E_*):=\frac{\pardeg_\Omega(E_*)}{\rk E}.
\end{equation}
\end{definition}

A coherent subsheaf \(\cV\subset E\) is \emph{saturated} if \(E/\cV\) is torsion free.
Outside an analytic subset of codimension at least two, it is a holomorphic subbundle of
\(E\). Its determinant extends uniquely across that subset as a line bundle, denoted by
\(\det\cV\), and \(c_1(\cV):=c_1(\det\cV)\). A saturated subsheaf is determined by its
restriction to the complement of a proper analytic subset. These facts are collected in
\cite[Ch.~2, \S1.1]{OSS}.

Near a general point of \(D_i\), the restriction of \(\cV\) is a subbundle of
\(E|_{D_i}\) and inherits the filtration by intersection with \(F^i_c(E_*)\); the
resulting parabolic structure is denoted by \(\cV_*\). We set
\begin{equation}\label{eq:deltai}
   \delta_i(\cV):=\sum_{c\in(0,1)}c\,\rk\Gr^i_c(\cV_*),
\end{equation}
where the ranks are constant away from a proper analytic subset of \(D_i\), and
\begin{equation}\label{eq:parc1V}
   \parc_1(\cV_*):=c_1(\cV)-\sum_i\delta_i(\cV)[D_i].
\end{equation}
The parabolic degree and slope of \(\cV_*\) are defined from
\eqref{eq:parc1V} by \eqref{eq:pardeg}.

\begin{definition}\label{def:stability}
A parabolic bundle \(E_*\) is \emph{stable} (resp.\ \emph{semistable})
with respect to \(\Omega\) if \(\mu_\Omega(\cV_*)<\mu_\Omega(E_*)\) (resp.\ \(\le\)) for
every saturated subsheaf \(\cV\subset E\) with \(0<\rk\cV<\rk E\). It is
\emph{polystable} if it is a direct sum of parabolic bundles which are
stable of the same parabolic slope; the direct sum is required to split every filtration
\(F^i_c\).
\end{definition}

\subsection{Uniformly adapted metrics}\label{ss:adapted}
For each \(i\), fix a smooth Hermitian metric \(h_i\) on \(\cO_X(D_i)\), let \(s_i\) be
the canonical section, and set
\[
   \rho_i:=|s_i|_{h_i}.
\]
After rescaling \(h_i\), one has \(0<\rho_i<1\) on \(X\setminus D_i\). For nonnegative
quantities on a specified set, write \(A\asymp B\) if \(C^{-1}B\le A\le CB\) there for
a uniform constant \(C\ge1\). For Hermitian forms the inequalities are understood in the
sense of quadratic forms. Constants denoted by \(C\) may change from line to line. In
adapted coordinates \eqref{eq:adap-coor}, \(\rho_j\asymp|z_j|\).

The compatibility between a Hermitian metric on \(E|_{\Xo}\) and the parabolic structure
used in Section~\ref{sec:technical} is uniform equivalence with an explicit diagonal
model in an adapted frame.

\begin{definition}\label{def:modelmetric}
Let \(E_*\) be a locally abelian parabolic bundle. On an adapted chart \(U\) as in
\eqref{eq:adap-coor}, with an adapted frame \(e_1,\dots,e_r\), the \emph{local model
metric} \(h_{\mathrm{mod},U}\) is the Hermitian metric on \(E|_{U\cap\Xo}\) for which this
frame is orthogonal and
\begin{equation}\label{eq:modelmetric}
   \Bigl|\sum_{\alpha=1}^{r}v^\alpha e_\alpha\Bigr|_{h_{\mathrm{mod},U}}^2
   :=\sum_{\alpha=1}^{r}
     \Bigl(\prod_{j=1}^{\ell}
       \rho_j^{-2\operatorname{wt}_{D_j}(e_\alpha)}\Bigr)|v^\alpha|^2 .
\end{equation}
\end{definition}

\begin{lemma}\label{lem:modelindependence}
The local model metrics obtained from different adapted frames, adapted coordinates, and
auxiliary metrics \(h_i\) are locally uniformly equivalent.
\end{lemma}

\begin{proof}
The functions \(\rho_i\) obtained from two smooth metrics on \(\cO_X(D_i)\) have a smooth
positive quotient, bounded above and below on a smaller chart. The same is true of
\(\rho_j/|z_j|\), so changes of adapted coordinates produce uniformly equivalent metrics on
a smaller chart.

It remains to compare adapted frames. Let \(f_1,\dots,f_r\) be another adapted frame and write
\(f_\gamma=\sum_\alpha A_{\alpha\gamma}e_\alpha\). If
\(\operatorname{wt}_{D_j}(e_\alpha)>
\operatorname{wt}_{D_j}(f_\gamma)\), then
\(f_\gamma|_{D_j}\in
F^j_{\operatorname{wt}_{D_j}(f_\gamma)}(E_*)\) by \eqref{eq:frameweight}, and
\eqref{eq:adaptedframe} gives \(z_j\mid A_{\alpha\gamma}\). Since all weights lie in
\([0,1)\), it follows that
\[
   |A_{\alpha\gamma}|\prod_{j=1}^{\ell}
      \rho_j^{-\operatorname{wt}_{D_j}(e_\alpha)}
   \le C\prod_{j=1}^{\ell}
      \rho_j^{-\operatorname{wt}_{D_j}(f_\gamma)}.
\]
Thus the change of frame is bounded between the two local model metrics. Applying the same
argument to its inverse gives the reverse bound.
\end{proof}

\begin{definition}\label{def:uniformlyadapted}
A Hermitian metric \(h\) on \(E|_{\Xo}\) is \emph{uniformly adapted} to \(E_*\) if every
point of \(D\) has an adapted chart \(U\) on which
\begin{equation}\label{eq:uniformlyadapted}
   h\asymp h_{\mathrm{mod},U}.
\end{equation}
\end{definition}

By Lemma~\ref{lem:modelindependence}, the choice of adapted frame, coordinates, and
auxiliary metrics in \eqref{eq:uniformlyadapted} is immaterial. If \(E_*\) is locally
abelian, patching the local model metrics of finitely many adapted charts covering \(D\)
with a smooth Hermitian metric on \(E\) by a partition of unity gives a uniformly adapted
metric on \(E|_{\Xo}\), and any two uniformly adapted metrics are uniformly equivalent on
\(\Xo\).

\begin{remark}
A uniformly adapted metric \(h\) is adapted to \(E_*\) in the growth sense of Mochizuki
\cite[\S3.5 and Definition~3.34]{Mochizuki}: by \eqref{eq:modelmetric} and
\eqref{eq:constituents}, together with the Riemann extension theorem applied to the
coefficients in an adapted frame, the holomorphic sections \(s\) of \(E\) over
\(U\cap\Xo\) with \(|s|_h=O\bigl(\prod_i\rho_i^{-c_i-\varepsilon}\bigr)\) near \(D\)
for every \(\varepsilon>0\) are exactly the sections of \(E_{\mathbf c}\).
\end{remark}

\subsection{Saturated subsheaves and Chern--Weil}
Let \(\cV\subset E\) be saturated of rank \(k\), let \(S\subset X\) be the singularity set
of \(E/\cV\), an analytic subset of codimension at least two, and let \(L:=\det\cV\). The
inclusion \(\Lambda^k\cV\to\Lambda^kE\) induces an injective morphism
\(\varsigma:L\to\Lambda^kE\), the \emph{determinant morphism}, which is nowhere zero on
\(X\setminus S\); in particular its zero locus has codimension at least two
\cite[Ch.~V, \S8]{Kobayashi}.

\begin{lemma}\label{lem:goodlocus}
There is a closed analytic subset \(Z\subset X\) of complex codimension at least two
such that on \(X\setminus Z\) the sheaf \(\cV\) is a holomorphic subbundle of \(E\),
the determinant morphism is nowhere zero, and every induced filtration step along
\(D_i\) is a subbundle of its generic rank.
\end{lemma}

\begin{proof}
Only finitely many distinct sheaves \(\Gr^i_c(\cV_*)\) occur. Each is coherent on
\(D_i\), so by semicontinuity of fibre dimension the set where its fibre dimension exceeds
its generic rank is a proper analytic subset of \(D_i\); let \(\Sigma_i\subset D_i\) be the
union of these finitely many sets, and put \(Z:=S\cup\bigcup_i\Sigma_i\). Each
\(\Sigma_i\) is a proper analytic subset of the divisor \(D_i\), hence an analytic subset
of \(X\) of complex codimension at least two, and \(S\) has codimension at least two; so
\(Z\) is a closed analytic subset of \(X\) of complex codimension at least two. Off
\(Z\) the sheaf \(\cV\) is a subbundle on which \(\varsigma\) is nowhere zero. Outside
\(\Sigma_i\) every graded sheaf \(\Gr^i_c(\cV_*)\) is locally free of its
generic rank, so the successive filtration steps along \(D_i\) are subbundles of their
generic ranks there.
\end{proof}

For a Hermitian holomorphic bundle \((E,h)\), write
\[
   F_h=\bar\partial(h^{-1}\partial h),
   \qquad
   c_1(E,h)=\frac{\mathrm{i}}{2\pi}\tr F_h
\]
for the curvature of the Chern connection and its first Chern form. For a line bundle
with holomorphic frame \(e\), this convention gives
\begin{equation}\label{eq:c1line}
   c_1(L,h)=-\frac{\mathrm{i}}{2\pi}\partial\bar\partial\log h(e,e).
\end{equation}
Given a K\"ahler form \(\omega\), set
\[
   \Phi_\omega:=\frac{\omega^{n-1}}{(n-1)!},
   \qquad
   dV_\omega:=\frac{\omega^n}{n!},
\]
and define \(\Lambda_\omega\) on \((1,1)\)-forms by
\(\alpha\wedge\Phi_\omega=(\Lambda_\omega\alpha)\,dV_\omega\).

Let \(\omega\) be a K\"ahler form on \(\Xo\), let \(h\) be a smooth Hermitian metric on
\(E|_{\Xo}\), and let \(\cV\subset E\) be saturated of rank \(k\), with determinant
morphism \(\varsigma:L\to\Lambda^kE\) and exceptional set \(Z\) as in
Lemma~\ref{lem:goodlocus}. On \(\Xo\setminus Z\) let \(\pi_\cV\) be the
\(h\)-orthogonal projection onto \(\cV\), let \(h_\cV:=\varsigma^*(\Lambda^kh)\) be the
induced metric on \(L\), and put \(\alpha_\cV:=c_1(L,h_\cV)\). The Chern--Weil identity
for a holomorphic subbundle \cite[Ch.~I, \S6]{Kobayashi} gives
\begin{equation}\label{eq:CWpointwise}
   \alpha_\cV\wedge\Phi_\omega
   =\frac1{2\pi}
    \Bigl[\tr\bigl(\pi_\cV\,\mathrm{i}\Lambda_\omega F_h\bigr)-|\bar\partial\pi_\cV|^2\Bigr]dV_\omega
   \qquad\text{on }\Xo\setminus Z .
\end{equation}
The metric \(h_\cV\) and the form \(\alpha_\cV\) are defined and smooth on the possibly
larger set \(\Xo\setminus\{\varsigma=0\}\), while \eqref{eq:CWpointwise} holds on
\(\Xo\setminus Z\). Integrated over a compact manifold, \eqref{eq:CWpointwise} together
with the next lemma is Simpson's Chern--Weil formula for the degree of a saturated
subsheaf \cite[Lemma~3.2]{Simpson}.

\begin{lemma}\label{lem:c1current}
Let \(\cV\subset E\) be saturated of rank \(k\), let \(h\) be smooth on \(E|_{\Xo}\) and
let \(q\) be a smooth metric on \(L=\det\cV\). Then \(\alpha_\cV\in L^1_{\mathrm{loc}}(\Xo)\),
\(\log(h_\cV/q)\in L^1_{\mathrm{loc}}(\Xo)\). If \([\alpha_\cV]\) denotes the current
defined by the locally integrable form \(\alpha_\cV\), then
\begin{equation}\label{eq:c1current}
   [\alpha_\cV]=c_1(L,q)
      -\frac{\mathrm{i}}{2\pi}\partial\bar\partial\log\frac{h_\cV}{q}
\end{equation}
as currents on \(\Xo\).
Moreover \(|\bar\partial\pi_\cV|^2\in L^1_{\mathrm{loc}}(\Xo)\).
\end{lemma}

\begin{proof}
Both assertions are local on \(\Xo\), where \(h\) is smooth. Let \(B\Subset\Xo\) be a
coordinate ball on which \(L\) and \(\Lambda^kE\) are trivial, let \(\xi\) be a
holomorphic frame of \(L\) and put \(s:=\varsigma(\xi)\in H^0(B,\Lambda^kE)\), with zero
set \(A\) of codimension at least two. In a holomorphic frame of \(\Lambda^kE\),
\(|s|^2_{\Lambda^kh}\asymp\sum_\alpha|s^\alpha|^2\) on \(B'\Subset B\), and
\(\log\sum_\alpha|s^\alpha|^2\) is plurisubharmonic and not identically \(-\infty\), hence
locally integrable \cite[Ch.~I, Thm.~4.17]{DemaillyBook}; so
\(\log(h_\cV/q)=\log|s|^2_{\Lambda^kh}-\log|\xi|^2_q\) is in \(L^1_{\mathrm{loc}}(B)\).
On \(B\setminus A\), the Cauchy--Schwarz inequality gives
\(\mathrm{i}\partial\bar\partial\log|s|^2_{\Lambda^kh}\ge-C\omega_{\mathrm{euc}}\), so
this smooth form has locally finite mass near \(A\); its trivial extension across \(A\) is
closed by the Skoda--El Mir theorem \cite[Ch.~III, Thm.~2.3]{DemaillyBook}, and its
difference with the distributional current
\(\mathrm{i}\partial\bar\partial\log|s|^2_{\Lambda^kh}\) on \(B\) is a closed current of
order zero supported on \(A\), which vanishes by the support theorem
\cite[Ch.~III, Cor.~2.11]{DemaillyBook} because \(A\) has real codimension at least four.
By \eqref{eq:c1line}, this is \eqref{eq:c1current} on \(B\); since \(Z\cap B\) has measure
zero, \([\alpha_\cV]\) is also the current associated with the form on \(B\setminus Z\).

Finally, \eqref{eq:CWpointwise} gives, as measures on \(B\setminus Z\),
\[
   \frac{1}{2\pi}|\bar\partial\pi_\cV|^2\,dV_\omega
   =\frac{1}{2\pi}\tr\bigl(\pi_\cV\,\mathrm{i}\Lambda_\omega F_h\bigr)dV_\omega
     -\alpha_\cV\wedge\Phi_\omega .
\]
The first term on the right is bounded, because \(h\) and \(\omega\) are smooth on \(B\)
and \(|\pi_\cV|_h\le\sqrt{\rk E}\), and the second is in \(L^1_{\mathrm{loc}}(B)\). The
left-hand side is a nonnegative measure on \(B\setminus Z\) dominated by a locally finite
one, so \(|\bar\partial\pi_\cV|^2\in L^1_{\mathrm{loc}}(B)\).
\end{proof}

\subsection{K\"ahler cone metrics}\label{sec:cone}

\begin{definition}\label{def:conemetric}
A K\"ahler metric \(\omega\) on \(\Xo\) is a \emph{cone metric} with cone angle
\(2\pi\beta_i\) along \(D_i\) if
\begin{enumerate}
\item[(C1)] there are a smooth K\"ahler form \(\omega_0\) on \(X\) and a function
\(\psi\in L^\infty(X)\cap C^\infty(\Xo)\), \(\omega_0\)-plurisubharmonic on \(X\), with
\begin{equation}\label{eq:globalpotential}
   \omega=\omega_0+\mathrm{i}\partial\bar\partial\psi \quad\text{on }\Xo;
\end{equation}
\item[(C2)] in every adapted chart \eqref{eq:adap-coor} the associated Hermitian metric
\(g\) is uniformly equivalent to the product metric \(g_{\boldsymbol\beta}\) in
\eqref{eq:productcone}, where \(\boldsymbol\beta=(\beta_1,\dots,\beta_\ell)\) and
\begin{equation}\label{eq:productcone}
   g_{\boldsymbol\beta}
   :=\sum_{j=1}^{\ell}|z_j|^{2\beta_j-2}|dz_j|^2
     +\sum_{j=\ell+1}^{n}|dz_j|^2 .
\end{equation}
\end{enumerate}
It is \emph{K\"ahler--Einstein} if \(\Ric(\omega)=\lambda\omega\) on \(\Xo\) for some
\(\lambda\in\R\), where \(\Ric(\omega):=-\mathrm{i}\partial\bar\partial\log\det(g_{p\bar q})\)
for \(\omega=\mathrm{i}g_{p\bar q}\,dz^p\wedge d\bar z^q\).
\end{definition}

\begin{remark}
	The K\"ahler--Einstein metrics constructed by Guenancia--P\u{a}un \cite{GuenanciaPaun} satisfy \textup{(C1)} and \textup{(C2)}.
\end{remark}

A cone metric is a K\"ahler form on \(\Xo\) only, and its K\"ahler class is read through
\textup{(C1)}: \(\Omega:=[\omega]\) means \([\omega_0]\in H^{1,1}(X,\R)\). This does not
depend on the choice of \(\omega_0\) and \(\psi\): if
\(\omega_0+\mathrm{i}\partial\bar\partial\psi=\omega_0'+\mathrm{i}\partial\bar\partial\psi'\)
on \(\Xo\) with both potentials bounded, then on a coordinate ball \(U\) choose a smooth
function \(f\) with \(\mathrm{i}\partial\bar\partial f=\omega_0-\omega_0'\); the function
\((\psi'-\psi)-f\) is bounded and pluriharmonic on \(U\setminus D\), hence extends
pluriharmonically across \(D\) \cite[Ch.~I, Thm.~5.24]{DemaillyBook}. Therefore
\(\psi'-\psi\) extends smoothly across \(D\), and
\(\omega_0-\omega_0'=\mathrm{i}\partial\bar\partial(\psi'-\psi)\) on \(X\), so
\([\omega_0]=[\omega_0']\). The trivial extension of \(\omega\) across \(D\) is the closed
positive current \(\widetilde{\omega}:=\omega_0+\mathrm{i}\partial\bar\partial\psi\) on
\(X\), whose class is \(\Omega\): since \(\psi\) is bounded, \(\widetilde{\omega}\) has
vanishing Lelong numbers and its Bedford--Taylor products charge no pluripolar set
\cite{BedfordTaylor}, in particular no mass on \(D\).

Suppose in addition that \(\Ric(\omega)=\lambda\omega\) on \(\Xo\). Fix a smooth positive
volume form \(dV\) on \(X\), with Ricci form
\(\Ric(dV):=-\mathrm{i}\partial\bar\partial\log(dV/dV_{\mathrm{euc}})\) in local
coordinates, a smooth closed form representing \(2\pi c_1(X)\), and set on \(\Xo\)
\[
   F:=\log\frac{\omega^n}{dV}+2\sum_iw_i\log\rho_i+\lambda\psi .
\]
By \textup{(C2)} and the boundedness of \(\psi\), the function \(F\) is bounded near
\(D\). Moreover, by \eqref{eq:c1line} and \eqref{eq:globalpotential},
\[
   \mathrm{i}\partial\bar\partial F
   =\Ric(dV)-\lambda\omega_0-2\pi\sum_iw_i\,c_1\bigl(\cO_X(D_i),h_i\bigr)
   \qquad\text{on }\Xo ,
\]
whose right-hand side is smooth on \(X\). Locally subtracting a smooth potential for this
form, the remainder is bounded and pluriharmonic on the complement of \(D\), hence
extends pluriharmonically across \(D\) \cite[Ch.~I, Thm.~5.24]{DemaillyBook}. Thus \(F\)
extends smoothly to \(X\), and
\[
   \omega^n=\frac{e^{F-\lambda\psi}\,dV}{\prod_i\rho_i^{2w_i}} .
\]
In particular the bounded potential \(\psi\) lies in the conical Monge--Amp\`ere setting
of Guenancia--P\u{a}un \cite[Theorem~B]{GuenanciaPaun}. Moreover
\(\log(\omega^n/dV)=F-\lambda\psi-\sum_iw_i\log\rho_i^2\) is integrable on \(X\), so
\(\Ric(\omega)=-\mathrm{i}\partial\bar\partial\log(\omega^n/dV)+\Ric(dV)\) defines a
current on \(X\), and the Poincar\'e--Lelong formula
\(\mathrm{i}\partial\bar\partial\log\rho_i^2=2\pi[D_i]-2\pi\,c_1(\cO_X(D_i),h_i)\)
\cite[Ch.~III, (2.15)]{DemaillyBook} gives, as currents on \(X\),
\begin{equation}\label{eq:Ric-current}
   \Ric(\omega)=\lambda\widetilde{\omega}+2\pi\sum_iw_i[D_i].
\end{equation}

The degree identity of Section~\ref{sec:technical} uses the cone-metric hypothesis only
through the following estimates.

\begin{lemma}\label{lem:basecone}
Let \(\omega\) be a cone metric, with \(\omega_0\) as in \textup{(C1)}. Then:
\begin{enumerate}
\item[(B1)] \(\Vol_\omega(\Xo):=\int_{\Xo}dV_\omega<\infty\);
\item[(B2)] \(\omega\ge c_0\omega_0\) on \(\Xo\) for some \(c_0>0\); consequently every
smooth differential form on \(X\) has bounded \(\omega\)-norm on \(\Xo\).
\end{enumerate}
Moreover, in an adapted chart \eqref{eq:adap-coor},
\begin{equation}\label{eq:conevolume}
   dV_\omega\asymp
   \Bigl(\prod_{j=1}^{\ell}|z_j|^{2\beta_j-2}\Bigr)dV_{\mathrm{euc}},
\end{equation}
where \(dV_{\mathrm{euc}}\) is Euclidean volume in these coordinates, and there is
\(C>0\) with
\begin{equation}\label{eq:dlogrhoi}
   |d\log\rho_j|_\omega^2\le C|z_j|^{-2\beta_j}
   \qquad (1\le j\le\ell).
\end{equation}
\end{lemma}

\begin{proof}
Formula \eqref{eq:conevolume} follows from \eqref{eq:productcone}, and each
factor \(r_j^{2\beta_j-1}dr_j\) is integrable because \(\beta_j>0\); with the compactness
of \(X\) this gives (B1).

For (B2), cover a neighbourhood of \(D\) by finitely many adapted charts \(U_1,\dots,U_m\) and
shrink them to \(U_a'\Subset U_a\) still covering a neighbourhood \(W\) of \(D\). On each
\(U_a'\), \eqref{eq:productcone} gives \(g\ge c_a\,g_{\mathrm{euc}}\) because
\(|z_j|^{2\beta_j-2}\ge1\) for \(|z_j|\le1\), and \(\omega_0\le C_a\,\omega_{\mathrm{euc}}\)
by smoothness; hence \(\omega\ge(c_a/C_a)\omega_0\) on \(U_a'\cap\Xo\). On the compact set
\(X\setminus W\subset\Xo\), the smallest eigenvalue of \(\omega\) with respect to
\(\omega_0\) has a positive lower bound. Taking the minimum of these finitely many bounds
gives \(\omega\ge c_0\omega_0\), which reverses on the cotangent bundle and bounds the
\(\omega\)-norm of smooth forms by their \(\omega_0\)-norm.

Finally \(d\log\rho_j=d\log|z_j|+O(1)\) in an adapted chart. The metric dual to the
\(j\)-th factor of \eqref{eq:productcone} has size \(|z_j|^{2-2\beta_j}\), so
\(|d\log|z_j||^2_\omega\le
C|z_j|^{2-2\beta_j}|z_j|^{-2}=C|z_j|^{-2\beta_j}\), and the
smooth error has bounded \(\omega\)-norm by (B2).
\end{proof}

\begin{lemma}\label{lem:cohomology}
Let \(\omega\) be a cone metric. For every smooth closed \((1,1)\)-form \(\eta\) on \(X\),
\[
   \int_{\Xo}\eta\wedge\Phi_\omega=\int_X\eta\wedge\frac{\Omega^{n-1}}{(n-1)!},
   \qquad
   \Vol_\omega(\Xo)=\int_X\frac{\Omega^n}{n!}.
\]
\end{lemma}

\begin{proof}
The integrals converge absolutely by (B1) and (B2). Since \(\psi\) is bounded, the mixed
Bedford--Taylor products of \(\widetilde{\omega}\) and \(\omega_0\) are well defined
closed positive currents on \(X\), with
\(\mathrm{i}\partial\bar\partial\psi\wedge S:=\mathrm{i}\partial\bar\partial(\psi S)\) for
closed positive \(S\) \cite[Ch.~III, \S3]{DemaillyBook}, and they put no mass on the
pluripolar set \(D\) \cite{BedfordTaylor}; so the integrals over \(\Xo\) equal
\(\int_X\eta\wedge\widetilde{\omega}^{\,n-1}\) and \(\int_X\widetilde{\omega}^{\,n}\).
Moreover \(\widetilde{\omega}^{\,n-1}-\omega_0^{n-1}=\mathrm{i}\partial\bar\partial\psi\wedge S\)
with \(S:=\sum_{j=0}^{n-2}\widetilde{\omega}^{\,j}\wedge\omega_0^{n-2-j}\) closed and
positive, and Stokes' theorem gives
\(\int_X\eta\wedge\mathrm{i}\partial\bar\partial(\psi S)
=\int_X\psi\,\mathrm{i}\partial\bar\partial\eta\wedge S=0\). The same argument with
\(\widetilde{\omega}^{\,n}-\omega_0^{n}=\mathrm{i}\partial\bar\partial\psi\wedge
\sum_{j=0}^{n-1}\widetilde{\omega}^{\,j}\wedge\omega_0^{n-1-j}\) gives the volume identity.
\end{proof}

By Lemma~\ref{lem:cohomology}, integrals over \(\Xo\) of smooth closed forms against
\(\omega^{n-1}\), and \(\Vol_\omega(\Xo)\), are computed by \(\Omega\); this is what
justifies the notation \(\Omega=[\omega]\).

\section{The polystability criterion}\label{sec:technical}

The main result of this section is the following polystability criterion.

\begin{theorem}\label{thm:HEcriterion}
Let \(E_*\) be a locally abelian parabolic bundle on \((X,D)\) such that, along each
\(D_i\), the only possible positive weight is \(w_i\), with
\begin{equation}\label{eq:oneweight}
   \rk\Gr^i_{w_i}(E_*)\le1 .
\end{equation}
Let \(\omega\) be a cone metric with \(\Omega=[\omega]\), and let \(h\) be a Hermitian
metric on \(E|_{\Xo}\), uniformly adapted to \(E_*\), such that
\begin{equation}\label{eq:HEgeneral}
	\mathrm{i}\Lambda_\omega F_h=\mu\Id_E , \qquad \mu \in \R .
\end{equation}
Then \(E_*\) is parabolic polystable with respect to \(\Omega\), of slope \(\mu_\Omega(E_*)=\mu\Vol_\omega(\Xo)/2\pi\).
\end{theorem}

The proof uses the next theorem, which is the main analytic result of the paper.

\begin{theorem}\label{thm:degree}
Under the hypotheses of Theorem~\ref{thm:HEcriterion}, for every saturated subsheaf \(\cV\subset E\) of
rank \(k\ge1\) the integral \(\int_{\Xo}|\bar\partial\pi_\cV|^2dV_\omega\) is finite and
\begin{equation}\label{eq:degreeformula}
   \pardeg_\Omega(\cV_*)
   =\frac{\mu k}{2\pi}\Vol_\omega(\Xo)
    -\frac1{2\pi}\int_{\Xo}|\bar\partial\pi_\cV|^2\,dV_\omega .
\end{equation}
\end{theorem}

Related degree comparisons have been proved under different hypotheses. Biquard relates
metric growth, parabolic extensions and Chern--Weil degrees in complex dimension two
\cite[Theorem~1.1 and p.~314]{BiquardParabolic}, and treats logarithmic Higgs bundles over
a Poincar\'e type metric \cite[Proposition~7.2 and Lemma~8.5]{BiquardHiggs}. Li constructs
a reference metric along normal crossing divisors and proves the comparison for that
metric and for compatible metrics
\cite[Definition~5.1 and Propositions~5.5, 5.7]{Li}. Mochizuki develops the theory of
tame harmonic bundles \cite{Mochizuki}, and Jiang--Li prove a Kobayashi--Hitchin
correspondence for saturated reflexive parabolic sheaves on compact K\"ahler manifolds
\cite{JiangLi}.

These results do not directly apply to the metric induced by a K\"ahler--Einstein cone
metric. Li's compatibility condition requires, besides mutual boundedness with a
reference metric, \(L^2\)-control of the first derivative of the endomorphism relating
the two metrics. Jiang--Li's compatibility condition includes the current inequality
which is at issue here, while their admissibility requires \(F_h\in L^2\). In contrast,
the hypotheses of Theorem~\ref{thm:degree} control only the local model of \(h\) and the
contraction \(\mathrm{i}\Lambda_\omega F_h\). The proof replaces derivative and
full-curvature control by direct estimates for a renormalized determinant: the
determinant metric of \(\cV\) is divided by its prescribed divisorial growth, the
logarithm of the quotient is bounded above and, one component at a time, below, and is
shown to be integrable against the cone volume form; a logarithmic cut-off then removes
any boundary contribution along \(D\). 

\subsection{The renormalized determinant}\label{ss:renormalized}
Throughout Sections~\ref{ss:renormalized}--\ref{ss:proofdegree}, \(E_*\), \(\omega\)
and \(h\) are as in Theorem~\ref{thm:degree}, and \(\cV\subset E\) is a fixed saturated
subsheaf of rank \(k\ge1\). Let \(Z\) be supplied by Lemma~\ref{lem:goodlocus}, let
\(L=\det\cV\), let \(\varsigma:L\to\Lambda^kE\) be the determinant morphism, choose an
auxiliary smooth Hermitian metric \(q\) on \(L\), and set
\[
   \eta:=c_1(L,q)-\sum_{i=1}^{N}\delta_i(\cV)\,c_1(\cO_X(D_i),h_i),
\]
so that \([\eta]=\parc_1(\cV_*)\) by \eqref{eq:parc1V}.

By \eqref{eq:oneweight}, the induced filtration of \(\cV_*\) along \(D_i\) has no
weight other than \(0\) and \(w_i\), and its positive-weight graded piece has rank
\(\epsilon_i\in\{0,1\}\); hence \(\delta_i(\cV)=\epsilon_iw_i\).
We now remove from \(h_\cV\) its prescribed parabolic growth. On \(\Xo\setminus Z\) define
the renormalized determinant
\[
   \Psi:=\Bigl(\prod_{i=1}^{N}\rho_i^{2\delta_i(\cV)}\Bigr)\frac{h_\cV}{q},
   \qquad u:=\log\Psi ,
\]
so that \(u\) is the logarithm of the determinant metric with the expected divisorial
singularities factored out. Replacing \(q\) by another smooth metric changes \(u\) by a
smooth function.
By \eqref{eq:c1line} and Lemma~\ref{lem:c1current},
\begin{equation}\label{eq:ddcu}
   \mathrm{i}\partial\bar\partial u
   =2\pi\bigl(\eta-[\alpha_\cV]\bigr)
   \qquad\text{as currents on }\Xo .
\end{equation}

\subsection{Determinant estimates}
In an adapted chart \eqref{eq:adap-coor} with adapted frame \(e_1,\dots,e_r\), write
\(e_I=e_{\alpha_1}\wedge\dots\wedge e_{\alpha_k}\) for
\(I=\{\alpha_1<\dots<\alpha_k\}\) and
\(\operatorname{wt}_{D_j}(e_I):=\sum_{\alpha\in I}\operatorname{wt}_{D_j}(e_\alpha)\).
Since exterior powers preserve inequalities between positive Hermitian forms, taking the
\(k\)-th exterior power of \eqref{eq:uniformlyadapted} and of the diagonal model
\eqref{eq:modelmetric} gives, for \(\Theta=\sum_I\Theta^Ie_I\),
\begin{equation}\label{eq:exteriormodel}
   |\Theta|_{\Lambda^kh}^2\asymp
   \sum_I\Bigl(\prod_{j=1}^{\ell}
      \rho_j^{-2\operatorname{wt}_{D_j}(e_I)}\Bigr)|\Theta^I|^2 .
\end{equation}

Fix an adapted chart \(U\), shrink it so that \(L\) is trivialized on a neighbourhood of
\(\overline U\) by a holomorphic section \(\xi\), and put
\(\Theta:=\varsigma(\xi)=\sum_I\Theta^Ie_I\). Since \(|\xi|_q\asymp1\),
\eqref{eq:exteriormodel} gives
\begin{equation}\label{eq:PsiLocal}
   \Psi\asymp\sum_I\Bigl(\prod_{j=1}^{\ell}
      \rho_j^{2(\delta_j(\cV)-\operatorname{wt}_{D_j}(e_I))}\Bigr)
      |\Theta^I|^2
   \qquad\text{on }U\cap(\Xo\setminus Z).
\end{equation}

\begin{lemma}\label{lem:upperdet}
The function \(u\) is bounded above on \(U\cap(\Xo\setminus Z)\).
\end{lemma}

\begin{proof}
For every \(j\le\ell\), the weight of \(e_I\) is either \(0\) or \(w_j\). If
\(\epsilon_j=1\), then
\(\delta_j(\cV)-\operatorname{wt}_{D_j}(e_I)\ge0\), so the \(j\)-th factor in
\eqref{eq:PsiLocal} is bounded. Suppose that \(\epsilon_j=0\). At a general point of
\(D_j\), the image of \(\Lambda^k\cV\) lies in \(\Lambda^kF^j_0(E_*)\). Therefore
\(\Theta^I|_{D_j}=0\) whenever \(\operatorname{wt}_{D_j}(e_I)=w_j\), and hence, by
\eqref{eq:adap-coor}, \(z_j\mid\Theta^I\) on \(U\). The corresponding factor is bounded
because \(|z_j|^2\rho_j^{-2w_j}\asymp\rho_j^{2(1-w_j)}\).
Applying this argument to every component meeting \(U\) shows that every summand in
\eqref{eq:PsiLocal} is bounded.
\end{proof}


The lower bound on \(u\) is established one component at a time.

\begin{lemma}\label{lem:lowerdet}
After shrinking \(U\), for every \(j\le\ell\) there are a holomorphic function \(F_j\) on a
neighbourhood of \(\overline U\) with \(F_j|_{D_j\cap U}\not\equiv0\), and constants
\(A,C>0\), such that
\begin{equation}\label{eq:lowerdet}
   u\ge\log|F_j|^2-C-A
      \sum_{\substack{1\le m\le\ell\\m\neq j}}\bigl|\log\rho_m\bigr|
   \qquad\text{on }U\cap(\Xo\setminus Z).
\end{equation}
\end{lemma}

\begin{proof}
Choose \(x\in(D_j\cap U)\setminus Z\). By
Lemma~\ref{lem:goodlocus}, near \(x\) the sheaf \(\cV\) is a subbundle of \(E\) and
\(\cV|_{D_j}\cap F^j_0(E_*)\) is a subbundle of \(\cV|_{D_j}\) of rank \(k-\epsilon_j\).
Choose a local frame \(v_1,\dots,v_k\) of \(\cV\) near \(x\) whose first
\(k-\epsilon_j\) members restrict on \(D_j\) to a frame of this subbundle, and write
\(v_\gamma=\sum_\alpha M_{\alpha\gamma}e_\alpha\) with \(M\) an \(r\times k\) matrix of
holomorphic functions. By \eqref{eq:adaptedframe}, \(F^j_0(E_*)\) is spanned on \(D_j\)
by the \(e_\alpha\) of weight zero along \(D_j\); hence \(M_{\alpha\gamma}(x)=0\) whenever
\(\gamma\le k-\epsilon_j\) and \(\operatorname{wt}_{D_j}(e_\alpha)=w_j\).

If \(\epsilon_j=0\), the \(k\) columns of \(M(x)\) are linearly independent and supported
on the zero-weight rows, so some \(k\times k\) minor of \(M(x)\) formed by zero-weight rows
is nonzero; let \(I_j\) be its set of rows. If \(\epsilon_j=1\), the first \(k-1\) columns
of \(M(x)\) are linearly independent and supported on the zero-weight rows, so some
\((k-1)\times(k-1)\) minor formed by a set \(I'\) of zero-weight rows and the first
\(k-1\) columns is nonzero; moreover, the entry of \(M(x)\) in the unique row \(\alpha_+\)
of weight \(w_j\) and the last column is nonzero, because \(v_k(x)\notin F^j_0(E_*)_x\).
Put \(I_j:=I'\cup\{\alpha_+\}\); the \(k\times k\) minor of \(M(x)\) with rows \(I_j\) is
block triangular with nonzero diagonal blocks, hence nonzero. In either case
\[
   \det M_{I_j}(x)\ne0,
   \qquad
   \operatorname{wt}_{D_j}(e_{I_j})=\epsilon_jw_j=\delta_j(\cV),
\]
where \(M_{I_j}\) denotes the \(k\times k\) submatrix of \(M\) with rows \(I_j\).

Near \(x\), the frame \(\xi\) of \(L\) equals \(f\,v_1\wedge\dots\wedge v_k\) with \(f\)
holomorphic and nowhere zero, so \(\Theta=f\sum_I\det M_I\,e_I\) and
\(F_j:=\Theta^{I_j}\) satisfies \(F_j(x)\ne0\); in particular
\(F_j|_{D_j\cap U}\not\equiv0\). Thus \(F_j\) is a Pl\"ucker coefficient which is
generically nonzero along \(D_j\) and whose \(D_j\)-weight realizes the induced parabolic
weight \(\delta_j(\cV)\). Keeping only the summand
\(I=I_j\) in \eqref{eq:PsiLocal},
\[
   \Psi\ge c\,|F_j|^2\prod_{\substack{1\le m\le\ell\\m\neq j}}
      \rho_m^{2(\delta_m(\cV)-\operatorname{wt}_{D_m}(e_{I_j}))},
\]
the factor for \(m=j\) being absent because
\(\operatorname{wt}_{D_j}(e_{I_j})=\delta_j(\cV)\). All the exponents
are bounded because the weights lie in \([0,1)\), and \eqref{eq:lowerdet} follows on taking
logarithms.
\end{proof}

\subsection{Weighted integrability}
\begin{lemma}\label{lem:weightedL1}
The function \(u\) is integrable on \(\Xo\) with respect to \(dV_\omega\). Moreover, for
every \(i\) there are \(\varepsilon_0,C>0\) with
\begin{equation}\label{eq:uL1log}
   \int_{\{\varepsilon<\rho_i<\varepsilon_0\}}|u|\,|d\log\rho_i|_\omega^2\,dV_\omega
   \le C\Bigl(1+\log\frac{\varepsilon_0}{\varepsilon}\Bigr)
\end{equation}
for all sufficiently small \(\varepsilon>0\).
\end{lemma}

\begin{proof}
Both statements are local near \(D\). On relatively compact subsets of \(\Xo\), the
function \(u\) is locally integrable by Lemma~\ref{lem:c1current}, while
\(|d\log\rho_i|_\omega\) is bounded; hence the required estimates hold there. Work on an
adapted chart \(U'\Subset U\) as in
Lemmas~\ref{lem:upperdet} and~\ref{lem:lowerdet}, and fix \(j\le\ell\). Write
\(z_j=te^{\mathrm{i}\vartheta}\), let \(y=(z_m)_{m\neq j}\) denote the remaining coordinates and
let \(d\nu_j\) be the measure
\(\bigl(\prod_{m\le\ell,\,m\neq j}|z_m|^{2\beta_m-2}\bigr)
dV_{\mathrm{euc}}(y)\), the cone measure in the remaining variables.
For some \(t_0>0\), we claim
\begin{equation}\label{eq:circlebound}
   \sup_{0<t<t_0}\ \int\!\!\int_0^{2\pi}
   \bigl|u\bigl(te^{\mathrm{i}\vartheta},y\bigr)\bigr|\,d\vartheta\,d\nu_j(y)<\infty .
\end{equation}
Writing \(u^+:=\max\{u,0\}\) and \(u^-:=\max\{-u,0\}\), one has \(u^+\le C\) by
Lemma~\ref{lem:upperdet}, and Lemma~\ref{lem:lowerdet} gives
\begin{equation}\label{eq:uminusbound}
   u^-\le C+A\sum_{\substack{1\le m\le\ell\\m\neq j}}
      |\log\rho_m|-\log|F_j|^2 .
\end{equation}
Here \(C\) has been enlarged so that the right-hand side is nonnegative, which is
possible because \(|F_j|\) is bounded on \(U'\). The function \(F_j(0,\cdot)=F_j|_{D_j}\) is holomorphic and
not identically zero, so \(Y:=\{y:F_j(0,y)=0\}\) is a proper analytic subset; for
\(y\notin Y\) the function \(z_j\mapsto\log|F_j(z_j,y)|^2\) is subharmonic and not
identically \(-\infty\), whence
\[
   \frac1{2\pi}\int_0^{2\pi}\log\bigl|F_j(te^{\mathrm{i}\vartheta},y)\bigr|^2d\vartheta
   \ \ge\ \log|F_j(0,y)|^2 .
\]
Moreover \(\log|F_j(0,\cdot)|\in L^p_{\mathrm{loc}}\) for every \(p<\infty\). Indeed, by
the Weierstrass preparation theorem, near any zero and after a linear change of
coordinates, \(F_j|_{D_j}\) is a nonvanishing factor times a monic polynomial
\(P(\zeta,y')\). After shrinking the neighbourhood, all roots of \(P(\cdot,y')\) lie in a
fixed disc. Factoring this polynomial for each fixed \(y'\), the estimate
\(\int_0^1|\log r|^p r\,dr<\infty\), followed by Fubini's theorem, gives the assertion.
The density of \(d\nu_j\), on the other hand, lies in
\(L^{p'}_{\mathrm{loc}}\) for some \(p'>1\), since
\(p'(2\beta_m-2)>-2\) for
\(p'<\min_m(1-\beta_m)^{-1}\) and every \(\beta_m>0\); H\"older's
inequality makes \(\log|F_j(0,\cdot)|^2\) integrable against \(d\nu_j\). The same holds for
each \(|\log\rho_m|\), \(1\le m\le\ell\), \(m\neq j\). Integrating
\eqref{eq:uminusbound} in \(\vartheta\) and then in \(y\) proves \eqref{eq:circlebound}.

By \eqref{eq:conevolume}, \(dV_\omega\asymp t^{2\beta_j-1}dt\,d\vartheta\,d\nu_j\), so
\eqref{eq:circlebound} gives
\(\int_{U'}|u|\,dV_\omega\le C\int_0^{t_0}t^{2\beta_j-1}dt<\infty\). By
\eqref{eq:dlogrhoi},
\(|d\log\rho_j|^2_\omega dV_\omega\le C\,t^{-1}dt\,d\vartheta\,d\nu_j\),
and \eqref{eq:uL1log} follows from \eqref{eq:circlebound} by integrating
\(t^{-1}dt\) over a comparable annulus. Here \(\rho_j\asymp t\), so
\(\{\varepsilon<\rho_j<\varepsilon_0\}\) is contained in an annulus whose two radii are
fixed multiples of \(\varepsilon\) and \(\varepsilon_0\), respectively.
\end{proof}

\subsection{The logarithmic cut-off}
Set \(s_D:=-\sum_{i=1}^N\log\rho_i\), so that \(s_D\to+\infty\) at \(D\) and, on
\(\Xo\), the form \(\mathrm{i}\partial\bar\partial s_D\) is the restriction of a
smooth form on \(X\).

\begin{lemma}\label{lem:cutoff}
There are \(\chi_T\in C_c^\infty(\Xo)\), \(0\le\chi_T\le1\), with \(\chi_T\uparrow1\)
pointwise and
\[
   \lim_{T\to\infty}\int_{\Xo}\chi_T\,
      \mathrm{i}\partial\bar\partial u\wedge\Phi_\omega=0 .
\]
\end{lemma}

\begin{proof}
Fix \(T_0>0\) and a nonincreasing \(\gamma\in C^\infty(\R,[0,1])\) with \(\gamma\equiv1\)
on \((-\infty,0]\) and \(\gamma\equiv0\) on \([1,\infty)\). For \(T>T_0\) put
\(R_T:=T-T_0\) and \(\chi_T:=\gamma\bigl((s_D-T_0)/R_T\bigr)\); since \(X\) is compact and
\(s_D\to\infty\) at every component of \(D\), the support of \(\chi_T\) is compact in
\(\Xo\). Because \(u\in L^1_{\mathrm{loc}}(\Xo)\),
\(\mathrm{i}\partial\bar\partial u\) is a current on \(\Xo\) and
\(\Phi_\omega\) is a smooth closed form there, so
\[
   \int\chi_T\,\mathrm{i}\partial\bar\partial u\wedge\Phi_\omega
   =\int u\,\mathrm{i}\partial\bar\partial\chi_T\wedge\Phi_\omega .
\]
Moreover
\[
   \mathrm{i}\partial\bar\partial\chi_T
   =\frac{\gamma'}{R_T}\,\mathrm{i}\partial\bar\partial s_D
     +\frac{\gamma''}{R_T^2}\,\mathrm{i}\partial s_D\wedge\bar\partial s_D .
\]
The form \(\mathrm{i}\partial\bar\partial s_D\) has bounded \(\omega\)-norm by (B2),
so, by Lemma~\ref{lem:weightedL1}, the first contribution is bounded by
\[
   \frac{C}{R_T}\int_{\Xo}|u|\,dV_\omega=O(R_T^{-1}).
\]
For the second,
\(|ds_D|_\omega^2\le N\sum_i|d\log\rho_i|_\omega^2\). On \(\supp d\chi_T\) one has
\(s_D<T\), hence \(\rho_i>e^{-T}\) for every \(i\). Splitting each region
\(\{\rho_i<\varepsilon_0\}\) off from its complement and applying \eqref{eq:uL1log} with
\(\varepsilon=e^{-T}\), together with \(|d\log\rho_i|_\omega\le C\) on
\(\{\rho_i\ge\varepsilon_0\}\) and Lemma~\ref{lem:weightedL1},
\[
   \int_{\supp d\chi_T}|u|\,|ds_D|^2_\omega\,dV_\omega\le C(1+T).
\]
The second contribution is therefore
\(O\bigl((1+T)/R_T^2\bigr)=O(R_T^{-1})\).
\end{proof}

\subsection{Proof of the degree identity}\label{ss:proofdegree}
\begin{proof}[Proof of Theorem~\ref{thm:degree}]
By \eqref{eq:HEgeneral}, \(|\Lambda_\omega F_h|_h=|\mu|\sqrt{\rk E}\) is constant, so it is
integrable by (B1). Write \(\pi=\pi_\cV\). Substituting \eqref{eq:HEgeneral} into
\eqref{eq:CWpointwise},
\begin{equation}\label{eq:CWHE}
   \alpha_\cV\wedge\Phi_\omega
   =\frac1{2\pi}\bigl(\mu k-|\bar\partial\pi|^2\bigr)dV_\omega
   \qquad\text{on }\Xo\setminus Z,
\end{equation}
and \(|\bar\partial\pi|^2\in L^1_{\mathrm{loc}}(\Xo)\) by Lemma~\ref{lem:c1current}.
Combining \eqref{eq:ddcu} and \eqref{eq:CWHE} with the cut-off of
Lemma~\ref{lem:cutoff},
\begin{equation}\label{eq:threeTerms}
   \int_{\Xo}\chi_T\,\mathrm{i}\partial\bar\partial u\wedge\Phi_\omega
   =2\pi\int_{\Xo}\chi_T\,\eta\wedge\Phi_\omega
     -\mu k\int_{\Xo}\chi_T\,dV_\omega
     +\int_{\Xo}\chi_T|\bar\partial\pi|^2\,dV_\omega .
\end{equation}
Now let \(T\to\infty\). The left-hand side vanishes by Lemma~\ref{lem:cutoff}. On the
right, (B2), (B1) and dominated convergence apply to the first two terms, which converge
to \(2\pi\int_{\Xo}\eta\wedge\Phi_\omega\) and to \(\mu k\Vol_\omega(\Xo)\); the first
of these limits equals \(2\pi\pardeg_\Omega(\cV_*)\) by Lemma~\ref{lem:cohomology}. In the third term the
integrand is nonnegative and \(\chi_T\uparrow1\), so monotone convergence makes it
increase to \(\int_{\Xo}|\bar\partial\pi|^2dV_\omega\), a priori a value in
\([0,\infty]\). Since the other three terms of \eqref{eq:threeTerms} have finite limits,
we must have \(\int_{\Xo}|\bar\partial\pi|^2dV_\omega<\infty\), and
\eqref{eq:degreeformula} follows.
\end{proof}

\subsection{Proof of the polystability criterion}\label{sec:HE}
When equality holds in \eqref{eq:degreeformula}, the orthogonal projection \(\pi_\cV\) is a
holomorphic idempotent on \(\Xo\setminus Z\). To turn it into a splitting of \(E_*\) we
need endomorphisms compatible with the parabolic structure.

\begin{definition}\label{def:parend}
\(P(E_*)\subset\End(E)\) is the subsheaf of endomorphisms \(A\) with
\(A\bigl(F^i_c(E_*)\bigr)\subset F^i_c(E_*)\) for all \(i\) and \(c\).
\end{definition}

\begin{lemma}\label{lem:parend}
Let \(E_*\) be a locally abelian parabolic bundle and let \(A\in\End(E)\). In an adapted
chart \eqref{eq:adap-coor} with adapted frame, write
\(A(e_\alpha)=\sum_\beta A_{\beta\alpha}e_\beta\). The following are equivalent:
\begin{enumerate}
\item[(i)] \(A\in P(E_*)\);
\item[(ii)] \(z_j\mid A_{\beta\alpha}\) whenever
\(\operatorname{wt}_{D_j}(e_\beta)>
\operatorname{wt}_{D_j}(e_\alpha)\), for every adapted frame and
\(1\le j\le\ell\);
\item[(iii)] \(A(E_{\mathbf c})\subset E_{\mathbf c}\) for every \(\mathbf c\in[0,1)^N\).
\end{enumerate}
In particular \(P(E_*)\) is a sheaf of \(\cO_X\)-algebras. Moreover, if \(A\in P(E_*)\)
and \(\det A\) vanishes nowhere, then \(A^{-1}\in P(E_*)\).
\end{lemma}

\begin{proof}
(i)\(\Leftrightarrow\)(ii). By \eqref{eq:adaptedframe}, \(A\) preserves
\(F^j_c\) for all \(c\) if and only if \(A_{\beta\alpha}|_{D_j}=0\) whenever
\(\operatorname{wt}_{D_j}(e_\beta)>c\ge
\operatorname{wt}_{D_j}(e_\alpha)\) for some \(c\). Taking
\(c=\operatorname{wt}_{D_j}(e_\alpha)\) gives (ii).

(ii)\(\Rightarrow\)(iii). With \(m_\alpha=m_\alpha(\mathbf c)\) as in
\eqref{eq:constituents}, we must show \(m_\beta\mid m_\alpha A_{\beta\alpha}\). Let
\(j\le\ell\) satisfy
\(c_j<\operatorname{wt}_{D_j}(e_\beta)\). If
\(c_j<\operatorname{wt}_{D_j}(e_\alpha)\), then \(z_j\mid m_\alpha\);
otherwise
\(\operatorname{wt}_{D_j}(e_\beta)>c_j\ge
\operatorname{wt}_{D_j}(e_\alpha)\), and \(z_j\mid A_{\beta\alpha}\).

(iii)\(\Rightarrow\)(ii). Fix \(j\le\ell\) and a pair \((\beta,\alpha)\) with
\(\operatorname{wt}_{D_j}(e_\beta)>
\operatorname{wt}_{D_j}(e_\alpha)\). Choose \(\mathbf c\) with
\(c_j=\operatorname{wt}_{D_j}(e_\alpha)\) and \(c_{j'}\) larger than all weights
for \(1\le j'\le\ell\), \(j'\neq j\). Then \(m_\alpha=1\) and
\(m_\beta=z_j\), so (iii) forces \(z_j\mid A_{\beta\alpha}\).

Finally, let \(A\in P(E_*)\) with \(\det A\) nowhere zero. In the basis
\(\{m_\alpha e_\alpha\}\) of \(E_{\mathbf c}\) the endomorphism \(A|_{E_{\mathbf c}}\) has
matrix \((m_\alpha m_\beta^{-1}A_{\beta\alpha})\), of determinant \(\det A\); hence
\(A(E_{\mathbf c})=E_{\mathbf c}\) and \(A^{-1}(E_{\mathbf c})=E_{\mathbf c}\) for every
\(\mathbf c\). By (iii)\(\Rightarrow\)(i), \(A^{-1}\in P(E_*)\).
\end{proof}

\begin{lemma}\label{lem:projectionextension}
Let \(E_*\) be a locally abelian parabolic bundle and let \(h\) be uniformly adapted to
\(E_*\).
Let \(\pi\in H^0(\Xo,\End E)\) satisfy \(\pi^2=\pi=\pi^{*h}\). Then \(\pi\) extends
uniquely to a holomorphic endomorphism of \(E\) on \(X\), and the extension lies in
\(P(E_*)\).
\end{lemma}

\begin{proof}
Two holomorphic extensions agree on the dense open set \(\Xo\), which proves uniqueness.
Work in an adapted chart \(U\) as in \eqref{eq:adap-coor}, with adapted frame, and write
\(\pi(e_\alpha)=\sum_\beta\pi_{\beta\alpha}e_\beta\). An orthogonal projection satisfies
\(|\pi(e_\alpha)|_h\le|e_\alpha|_h\), so
\eqref{eq:modelmetric} and \eqref{eq:uniformlyadapted} give
\begin{equation}\label{eq:projectioncoeff}
   |\pi_{\beta\alpha}|\le C\prod_{j=1}^{\ell}
      \rho_j^{\,\operatorname{wt}_{D_j}(e_\beta)
         -\operatorname{wt}_{D_j}(e_\alpha)} .
\end{equation}
All exponents lie in \((-1,1)\), so \(\bigl(\prod_{j\le\ell}z_j\bigr)\pi_{\beta\alpha}\) is
bounded near \(D\). The Riemann extension theorem extends this holomorphic function from
\(U\setminus D\) to a holomorphic function \(G\) on \(U\)
\cite[Ch.~I, Cor.~5.25]{DemaillyBook}. It satisfies
\(|G|\le C\prod_{j=1}^{\ell}|z_j|^{1+\operatorname{wt}_{D_j}(e_\beta)
-\operatorname{wt}_{D_j}(e_\alpha)}\), and each exponent is positive, so
\(G\) vanishes on every \(D_j\cap U=\{z_j=0\}\) and is therefore divisible by
\(z_1\cdots z_\ell\).
Hence \(\pi_{\beta\alpha}=G/(z_1\cdots z_\ell)\) is holomorphic on \(U\).

Suppose \(\operatorname{wt}_{D_j}(e_\beta)>
\operatorname{wt}_{D_j}(e_\alpha)\). Away from the other components,
\eqref{eq:projectioncoeff} gives \(\pi_{\beta\alpha}\to0\) as \(z_j\to0\), so
\(\pi_{\beta\alpha}\) vanishes on a dense open subset of \(D_j\cap U\) and hence on all of
it; thus \(z_j\mid\pi_{\beta\alpha}\). By Lemma~\ref{lem:parend}, \(\pi\in P(E_*)\).
\end{proof}

\begin{lemma}\label{lem:directsummands}
Let \(E_*\) be a locally abelian parabolic bundle and let \(\pi\in P(E_*)\) satisfy
\(\pi^2=\pi\). Put
\(E_1:=\im\pi\), \(E_2:=\ker\pi\) and give them the induced parabolic structures
\(F^i_c(E_{s,*}):=F^i_c(E_*)\cap E_s|_{D_i}\). Then \(E_1\) and \(E_2\) are holomorphic
subbundles of \(E\), their induced parabolic structures are locally abelian, and
\begin{equation}\label{eq:par-splitting}
   F^i_c(E_*)=F^i_c(E_{1,*})\oplus F^i_c(E_{2,*})
   \qquad\text{for all }i,c .
\end{equation}
In particular \(E_*=E_{1,*}\oplus E_{2,*}\) as parabolic bundles.
\end{lemma}

\begin{proof}
\emph{Splitting the filtrations.}
A holomorphic idempotent has locally free image and kernel and \(E=E_1\oplus E_2\). Since
\(\pi\) and \(\Id-\pi\) lie in \(P(E_*)\), each preserves every \(F^i_c\), and
\(F^i_c=\pi(F^i_c)\oplus(\Id-\pi)(F^i_c)\); moreover
\(\pi(F^i_c)=F^i_c\cap E_1|_{D_i}=F^i_c(E_{1,*})\), because a vector of
\(F^i_c\cap E_1|_{D_i}\) is fixed by \(\pi\). As \(\pi|_{F^i_c}\) is an idempotent
endomorphism of a vector bundle, its image is a subbundle. This proves
\eqref{eq:par-splitting}; it remains to produce adapted frames.

\emph{Adapted frames.}
Fix an adapted chart \eqref{eq:adap-coor} centred at a point \(o\) of \(D\), with adapted
frame \(e_1,\dots,e_r\) and multiweights
\[
   \operatorname{wt}(e_\alpha)
   :=\bigl(\operatorname{wt}_{D_j}(e_\alpha)\bigr)_{j=1}^{\ell} \in[0,1)^\ell ;
\]
all statements below are about germs at \(o\). Let \(W\subset[0,1)^\ell\) be the set of
occurring multiweights, partially ordered by \(\mathbf a\le\mathbf b\) if \(a_j\le b_j\)
for all \(j\), let \(E^{(\mathbf a)}\) be the free \(\cO\)-submodule spanned by the
\(e_\alpha\) with \(\operatorname{wt}(e_\alpha)=\mathbf a\), so that
\(E=\bigoplus_{\mathbf a\in W}E^{(\mathbf a)}\), and write
\(A=(A^{\mathbf b\mathbf a})\) for the resulting block decomposition of \(A\in\End(E)\).
By Lemma~\ref{lem:parend}, \(A\in P(E_*)\) if and only if
\(A^{\mathbf b\mathbf a}\equiv0\bmod z_j\) whenever \(b_j>a_j\); in particular
\(A^{\mathbf b\mathbf a}(o)=0\) unless \(\mathbf b\le\mathbf a\).

Since \(\pi^{\mathbf a\mathbf c}(o)\pi^{\mathbf c\mathbf a}(o)\) can be nonzero only for
\(\mathbf a\le\mathbf c\le\mathbf a\), the \(\mathbf a\)-th diagonal block of
\(\pi(o)^2=\pi(o)\) reduces to
\(\bigl(\pi^{\mathbf a\mathbf a}(o)\bigr)^2=\pi^{\mathbf a\mathbf a}(o)\). A constant
change of basis inside each \(E^{(\mathbf a)}\) preserves multiweights, hence adaptedness,
so we may assume that every \(\pi^{\mathbf a\mathbf a}(o)\) is diagonal with entries \(0\)
and \(1\). Let \(\Pi\) be the constant block diagonal idempotent whose
\(\mathbf a\)-block is \(\pi^{\mathbf a\mathbf a}(o)\); being constant and block diagonal,
\(\Pi\in P(E_*)\). Set
\[
   Q:=\pi\Pi+(\Id-\pi)(\Id-\Pi),
\]
which lies in \(P(E_*)\) because \(P(E_*)\) is a sheaf of algebras. Its
\(\mathbf a\)-th diagonal block at \(o\) is
\[
   \pi^{\mathbf a\mathbf a}(o)\Pi^{\mathbf a\mathbf a}
   +\bigl(\Id-\pi^{\mathbf a\mathbf a}(o)\bigr)
    \bigl(\Id-\Pi^{\mathbf a\mathbf a}\bigr)
   =\bigl(\Pi^{\mathbf a\mathbf a}\bigr)^2
    +\bigl(\Id-\Pi^{\mathbf a\mathbf a}\bigr)^2=\Id,
\]
while \(Q^{\mathbf b\mathbf a}(o)=0\) unless \(\mathbf b\le\mathbf a\). Ordering \(W\)
compatibly with its partial order therefore makes \(Q(o)\) block triangular with identity
diagonal, so \(\det Q(o)=1\) and \(Q\) is invertible after shrinking the chart; by
Lemma~\ref{lem:parend}, \(Q^{-1}\in P(E_*)\). Finally \(\pi^2=\pi\) gives
\(\pi(\Id-\pi)=0\) and \(\Pi^2=\Pi\) gives \((\Id-\Pi)\Pi=0\), whence
\[
   \pi Q=\pi\Pi=Q\Pi .
\]

Both \(Q\) and \(Q^{-1}\) preserve every \(F^i_c\), so \(Q(F^i_c)=F^i_c\) and the frame
\(f_\alpha:=Q(e_\alpha)\) is again adapted, with the same multiweights. By
\(\pi Q=Q\Pi\), the isomorphism \(Q\) carries \(\im\Pi\) onto \(E_1\) and \(\ker\Pi\) onto
\(E_2\). As \(\Pi\) is diagonal with entries \(0\) and \(1\), those \(f_\alpha\) with
\(\Pi(e_\alpha)=e_\alpha\) form a frame of \(E_1\) and the remaining ones a frame of
\(E_2\); by \eqref{eq:par-splitting}, both frames are adapted. Hence \(E_{1,*}\) and
\(E_{2,*}\) are locally abelian. Away from \(D\) there is nothing to prove.
\end{proof}

\begin{proof}[Proof of Theorem~\ref{thm:HEcriterion}]
We induct on \(r=\rk E\). Let \(\cV\subset E\) be saturated with \(0<k=\rk\cV<r\). Taking
\(\cV=E\) in Theorem~\ref{thm:degree} gives
\(\pardeg_\Omega(E_*)=\mu r\Vol_\omega(\Xo)/2\pi\), so
\(\mu_\Omega(E_*)=\mu\Vol_\omega(\Xo)/2\pi\); applied to \(\cV\), it gives
\begin{equation}\label{eq:slopeidentity}
   \mu_\Omega(\cV_*)
   =\frac{\mu}{2\pi}\Vol_\omega(\Xo)
    -\frac1{2\pi k}\int_{\Xo}|\bar\partial\pi_\cV|^2dV_\omega
   \ \le\ \mu_\Omega(E_*),
\end{equation}
which is semistability, with equality if and only if \(\bar\partial\pi_\cV=0\). If the
inequality in \eqref{eq:slopeidentity} is strict for every such \(\cV\) -- in particular
if \(r=1\), when no such \(\cV\) exists -- then \(E_*\) is stable and we are done.

Assume equality for some \(\cV\), and let \(Z\) be supplied by
Lemma~\ref{lem:goodlocus}. Then \(\pi_\cV\) is holomorphic on
\(\Xo\setminus Z\). Since \(Z\cap\Xo\) is an analytic set of codimension at least two,
the Hartogs extension theorem extends \(\pi_\cV\) holomorphically across it; see also
\cite[Ch.~II, Prop.~6.1 and Ch.~I, Cor.~5.25]{DemaillyBook}. The
identities \(\pi^2=\pi=\pi^{*h}\) persist across \(Z\cap\Xo\) by continuity.
Lemma~\ref{lem:projectionextension} extends
it to \(\pi\in P(E_*)\) with \(\pi^2=\pi\). By Lemma~\ref{lem:directsummands},
\(E_*=E_{1,*}\oplus E_{2,*}\) with \(E_1=\im\pi\), both summands locally abelian. Both
\(E_1\) and \(\cV\) are saturated subsheaves of \(E\) (indeed \(E/E_1\cong E_2\) is locally
free) and they agree on the complement of a proper analytic subset of \(X\); therefore \(E_1=\cV\).

The decomposition \(E=E_1\oplus E_2\) is \(h\)-orthogonal on \(\Xo\), so the Chern
connection and curvature split and \(\mathrm{i}\Lambda_\omega F_{h_s}=\mu\Id_{E_s}\) for
\(s=1,2\), where \(h_s\) is the restricted metric. By Lemma~\ref{lem:directsummands} we may
choose adapted frames of \(E_1\) and \(E_2\) whose union is an adapted frame of \(E\); in
such a frame the restriction of the local model metric \eqref{eq:modelmetric} of \(E_*\)
is the local model metric of \(E_{s,*}\), so \(h_s\) is uniformly adapted to
\(E_{s,*}\). The splitting of the filtrations also shows that \eqref{eq:oneweight} is
inherited by each summand, so the induction hypothesis applies to \(E_{s,*}\), \(h_s\).
Applying Theorem~\ref{thm:degree} to \(E_s\) gives
\[
   \mu_\Omega(E_{s,*})
   =\frac{\mu}{2\pi}\Vol_\omega(\Xo)
   =\mu_\Omega(E_*)
   \qquad (s=1,2).
\]
Hence \(E_*\) is a direct sum of stable parabolic bundles of slope
\(\mu_\Omega(E_*)\), splitting every filtration.
\end{proof}

\section{Proof of Theorem A}\label{sec:thmA}

\subsection{The parabolic tangent bundle}

We introduce a natural parabolic structure on \(TX\) associated to the pair \((X,\Delta)\). We refer to it as the \emph{parabolic tangent bundle}.

\begin{definition}\label{def:parTX}
\(TX_*\) is the parabolic structure on \(TX\) given along \(D_i\) by
\begin{equation}\label{eq:TXfiltration}
   F^i_c(TX_*)=
   \begin{cases}
      TD_i,& 0\le c<w_i,\\
      TX|_{D_i},& w_i\le c<1 .
   \end{cases}
\end{equation}
\end{definition}

In adapted coordinates \eqref{eq:adap-coor}, the vectors \(\partial_{z_\alpha}\) with
\(\alpha\neq j\) restrict to a frame of \(TD_j\) along \(D_j=\{z_j=0\}\), while the class
of \(\partial_{z_j}\) spans \(N_j:=TX|_{D_j}/TD_j\). Thus the coordinate frame
\(\partial_{z_1},\dots,\partial_{z_n}\) splits every filtration \eqref{eq:TXfiltration}
simultaneously, with weights
\begin{equation}\label{eq:tangentmultiweights}
   \operatorname{wt}_{D_j}(\partial_{z_\alpha})
   =\begin{cases}w_j,&\alpha=j,\\0,&\alpha\neq j,\end{cases}
   \qquad 1\le j\le\ell,\quad 1\le\alpha\le n ,
\end{equation}
so \(TX_*\) is locally abelian and \(\partial_{z_1},\dots,\partial_{z_n}\) is an adapted
frame. Along \(D_i\) the only positive weight is \(w_i\), with
\(\Gr^i_{w_i}(TX_*)=N_i\) of rank one, so \eqref{eq:parc1} and \eqref{eq:logc1} give
\[
   \parc_1(TX_*)=c_1(X)-\sum_iw_i[D_i]=c_1(X,\Delta).
\]

The zero constituent of \(TX_*\) is the logarithmic tangent bundle. The
\emph{logarithmic tangent bundle} \(\Tlog\subset TX\) is the locally free sheaf of
holomorphic vector fields tangent to every component of \(D\); in adapted coordinates
\eqref{eq:adap-coor} it has the frame
\begin{equation}\label{eq:logtangentframe}
   z_1\partial_{z_1},\dots,z_\ell\partial_{z_\ell},
   \partial_{z_{\ell+1}},\dots,\partial_{z_n}.
\end{equation}
Since \(F^i_0(TX_*)=TD_i\), formula \eqref{eq:constituentglobal} gives
\begin{equation}\label{eq:TXzero}
   (TX)_{\mathbf 0}=\Tlog ,
\end{equation}
and \eqref{eq:logtangentframe} is the frame \eqref{eq:constituents} of the zero
constituent. The logarithmic tangent bundle fits into the exact sequence
\[
   0\longrightarrow\Tlog\longrightarrow TX
   \longrightarrow\bigoplus_i(\iota_i)_*N_i\longrightarrow0 .
\]

\subsection{Proof of Theorem A}
Let \(\omega\) be a K\"ahler--Einstein cone metric as in Theorem~A, write
\(\Ric(\omega)=\lambda\omega\), and let \(g\) be the induced Hermitian metric on
\(TX|_{\Xo}\). For the Hermitian metric induced by a K\"ahler metric on the tangent
bundle, the mean curvature is the Ricci tensor \cite[Ch.~I, (7.23)]{Kobayashi}:
\[
   \mathrm{i}\Lambda_\omega F_g=g^{-1}\Ric(\omega),
\]
where \(g^{-1}\Ric(\omega)\) is the endomorphism with components
\(g^{r\bar t}\Ric_{s\bar t}\). Hence \(\mathrm{i}\Lambda_\omega F_g=\lambda\Id_{TX}\) on
\(\Xo\).

By \eqref{eq:productcone}, in adapted coordinates \eqref{eq:adap-coor}, the induced
metric of a cone metric satisfies, for \(v=\sum_\alpha v^\alpha\partial_{z_\alpha}\),
\begin{equation}\label{eq:tangentmodel}
   |v|_g^2\asymp\sum_{j=1}^{\ell}\rho_j^{-2w_j}|v^j|^2
   +\sum_{j=\ell+1}^{n}|v^j|^2 .
\end{equation}

\begin{lemma}\label{lem:tangentmodel}
The Hermitian metric \(g\) induced on \(TX|_{\Xo}\) by a cone metric is uniformly adapted
to \(TX_*\).
\end{lemma}

\begin{proof}
By \eqref{eq:tangentmultiweights}, the right-hand side of \eqref{eq:tangentmodel} is the
local model metric \eqref{eq:modelmetric} of \(TX_*\) in the adapted frame
\(\partial_{z_1},\dots,\partial_{z_n}\).
\end{proof}

By Lemma~\ref{lem:tangentmodel}, \(g\) is uniformly adapted to \(TX_*\), and
\(\mathrm{i}\Lambda_\omega F_g=\lambda\Id_{TX}\). Since \(TX_*\) is locally abelian with
the rank-one positive weight \(w_i\) along \(D_i\), Theorem~\ref{thm:HEcriterion} applies
and proves Theorem~A.\qed

\section{Proof of Theorem B}\label{sec:thmB}

\subsection{Tian's extension and its parabolic structure}
Let \(\Omega\) be a K\"ahler class on \(X\). Extensions of \(TX\) by \(\cO_X\) are
classified by \(H^1(X,\Hom(TX,\cO_X))=H^1(X,\Omega^1_X)\cong H^{1,1}(X,\C)\). We identify
\((1,1)\)-forms with elements of \(A^{0,1}(\Hom(TX,\cO))\) by
\[
   \sigma_{p\bar q}\,dz^p\wedge d\bar z^q
   \longmapsto
   \sigma_{p\bar q}\,d\bar z^q\otimes dz^p .
\]
Thus, for a \((1,0)\)-form \(\tau=\tau_p\,dz^p\), viewed as a section of
\(\Hom(TX,\cO)\),
\[
   \bar\partial\tau=(\partial_{\bar q}\tau_p)\,d\bar z^q\otimes dz^p .
\]
On the \(C^\infty\) bundle
\(\cO_X\oplus TX\), for a smooth closed real \((1,1)\)-form \(\sigma\) on an open set,
consider
\[
   \bar\partial_\sigma:=\begin{pmatrix}\bar\partial&\sigma\\0&\bar\partial\end{pmatrix},
\]
a \(\bar\partial\)-operator because \(\bar\partial\sigma=0\). The holomorphic bundle
\((\cO_X\oplus TX,\bar\partial_\sigma)\) sits in an exact sequence
\begin{equation}\label{eq:Tianextension}
   0\longrightarrow\cO_X\longrightarrow E \stackrel{p}{\longrightarrow}TX\longrightarrow0
\end{equation}
of class \([\sigma]\), and changing \(\sigma\) by a \(\bar\partial\)-exact form yields an
isomorphic extension. Following Tian \cite{Tian92}, let \(E\) be the extension
\eqref{eq:Tianextension} whose class is a non-zero multiple of \(\Omega\). For
\(c\in\C^*\) the map \(\operatorname{diag}(c^{-1},\Id)\) on \(\cO_X\oplus TX\)
intertwines \(\bar\partial_{c\sigma}\) with \(\bar\partial_\sigma\) and commutes with
\(p\), so the extensions of class \(c\Omega\), \(c\in\C^*\), are all isomorphic by
isomorphisms commuting with \(p\); \(E\) is thus determined by the line
\(\C\cdot\Omega\subset H^{1,1}(X,\C)\).

\begin{definition}\label{def:Tian-parabolic}
\(E_*\) is the parabolic structure on \(E\) given by
\begin{equation}\label{eq:Tianfiltration}
   F^i_c(E_*):=p^{-1}\bigl(F^i_c(TX_*)\bigr)\subset E|_{D_i}.
\end{equation}
\end{definition}

Since the isomorphisms between the extensions of class \(c\Omega\) commute with \(p\),
they preserve the filtration \eqref{eq:Tianfiltration}, and the parabolic bundle \(E_*\)
is independent of the multiple. Over an adapted chart \eqref{eq:adap-coor} the sequence
\eqref{eq:Tianextension} splits holomorphically, its quotient being locally free, and in
such a splitting \(F^i_c(E_*)=\cO_{D_i}\oplus F^i_c(TX_*)\). Hence \(E_*\) is locally
abelian with adapted frame \(1,\partial_{z_1},\dots,\partial_{z_n}\), the first vector of
multiweight zero and the remaining ones with the multiweights
\eqref{eq:tangentmultiweights}. In particular \(\Gr^i_{w_i}(E_*)=N_i\), and
\eqref{eq:Tianfiltration} gives
\[
   (E)_{\mathbf 0}=p^{-1}\bigl(\Tlog\bigr).
\]

\subsection{Proof of Theorem B}
Let \(\omega\) be a K\"ahler--Einstein cone metric as in Theorem~B, with \([\omega]=\Omega\)
and \(\Ric(\omega)=\lambda\omega\), \(\lambda>0\), and let \(g\) be the induced Hermitian
metric on \(TX|_{\Xo}\). Set
\begin{equation}\label{eq:tpositive}
   t:=\sqrt{\frac{\lambda}{n+1}},
\end{equation}
the normalization for which the two diagonal curvature blocks computed below agree, and
represent \(E\) by \((\cO_X\oplus TX,\bar\partial_{t\omega_0})\), with \(\omega_0\) as in
\textup{(C1)}; its class is \(t\Omega\).

\emph{A convenient representative.}
By \textup{(C1)}, \(\omega=\omega_0+\mathrm{i}\partial\bar\partial\psi\) on \(\Xo\),
with \(\psi\) bounded on \(X\) and smooth on \(\Xo\). On \(\Xo\) we represent the extension
by the K\"ahler--Einstein form itself, putting
\(\mathcal F:=(\cO\oplus TX,\bar\partial_{t\omega})\); the advantage of this representative
is that \(t\omega\) is parallel for the K\"ahler metric, which is what makes the curvature
computation transparent.

Set \(\tau:=t\,\mathrm{i}\partial\psi\), a smooth \((1,0)\)-form on \(\Xo\). Under
the identification of \((1,1)\)-forms with elements of \(A^{0,1}(\Hom(TX,\cO))\),
\(\bar\partial\tau\) corresponds to
\(t\,\mathrm{i}\partial\bar\partial\psi=t(\omega-\omega_0)\) by
\eqref{eq:globalpotential}. A direct
computation gives, for \(\upsilon:=\begin{psmallmatrix}1&\tau\\0&1\end{psmallmatrix}\),
\[
   \upsilon^{-1}\circ\bar\partial_{t\omega_0}\circ\upsilon
   =\bar\partial_{t\omega_0+\bar\partial\tau}=\bar\partial_{t\omega} ,
\]
so \(\upsilon:\mathcal F\to E|_{\Xo}\) is an isomorphism of holomorphic bundles.

\emph{The Hermitian--Einstein metric.}
Equip \(\mathcal F\) with \(h_\omega:=1\oplus g\). In this \(C^\infty\) splitting \(\cO\) is
a holomorphic subbundle of \(\mathcal F\) with flat induced metric, the quotient is
\((TX,g)\), and the splitting is \(h_\omega\)-orthogonal. Writing \(B:=t\omega\in
A^{0,1}(\Hom(TX,\cO))\), the Chern connection of \((\mathcal F,h_\omega)\) is
\(\begin{psmallmatrix}\nabla_{\cO}&B\\-B^*&\nabla_{TX}\end{psmallmatrix}\), and its
curvature has diagonal blocks \(-B\wedge B^*\) on \(\cO\) and \(F_g-B^*\wedge B\) on
\(TX\) \cite[Ch.~I, \S6]{Kobayashi}, \cite[\S3.1]{ChiLi}; the off-diagonal blocks
\(\pm\nabla B^{(*)}\) vanish because \(\omega\) is parallel. In a \(g\)-unitary coframe
\(\mathrm{i}\Lambda_\omega(B\wedge B^*)=-t^2n\) and
\(\mathrm{i}\Lambda_\omega(B^*\wedge B)=t^2\Id_{TX}\), so, with
\(\mathrm{i}\Lambda_\omega F_g=\lambda\Id_{TX}\) from Section~\ref{sec:thmA},
\[
   \mathrm{i}\Lambda_\omega F_{h_\omega}\big|_{\cO}=t^2n,
   \qquad
   \mathrm{i}\Lambda_\omega F_{h_\omega}\big|_{TX}=(\lambda-t^2)\Id_{TX} .
\]
Equation~\eqref{eq:tpositive} makes both equal to \(n\lambda/(n+1)\), so
\(\mathrm{i}\Lambda_\omega F_{h_\omega}=\frac{n\lambda}{n+1}\Id\).

\emph{Uniform adaptedness.}
Transport the metric: \(h:=(\upsilon^{-1})^*h_\omega\) on \(E|_{\Xo}\). Being the
image of \(h_\omega\) under a holomorphic isomorphism, it satisfies
\(\mathrm{i}\Lambda_\omega F_h=\frac{n\lambda}{n+1}\Id_{E}\). By the discussion in
Section~\ref{sec:cone} and \cite[Theorem~B]{GuenanciaPaun}, the bounded potential
\(\psi\) has the standard conical \(C^{2,\alpha,\beta}\) regularity; in particular, in
adapted coordinates \eqref{eq:adap-coor},
\[
   |z_j|^{1-\beta_j}|\partial_{z_j}\psi|\le C\quad(j\le\ell),
   \qquad
   |\partial_{z_j}\psi|\le C\quad(j>\ell),
\]
see also \cite[\S7.1, Definition~2 and Lemma~4]{GuenanciaPaun}. Together with
\textup{(C2)}, this gives \(|\partial\psi|_\omega\le C\). Hence
\(|\tau|_{h_\omega}=t|\partial\psi|_\omega\le C\), so \(\upsilon^{\pm1}\) are
uniformly bounded and \(h\asymp1\oplus g\) as Hermitian forms on \(\cO\oplus TX\).

Finally, cover a neighbourhood of \(D\) by finitely many adapted charts \(U_1,\dots,U_m\)
with \(U_a'\Subset U_a\) still covering a neighbourhood of \(D\), and fix a holomorphic
splitting of \eqref{eq:Tianextension} over each \(U_a\). The change of frame between it and
the \(C^\infty\) splitting is \(\begin{psmallmatrix}1&\tau_a\\0&1\end{psmallmatrix}\)
with \(\tau_a\) the smooth \((1,0)\)-form on \(U_a\) measuring the difference between
the chosen local holomorphic splitting of \eqref{eq:Tianextension} and the fixed global
\(C^\infty\) splitting \(\cO\oplus TX\); by \textup{(B2)},
\(|\tau_a|_\omega\le c_0^{-1/2}|\tau_a|_{\omega_0}\), which is bounded on
\(\overline{U_a'}\) by compactness. Taking the maximum over the finitely many charts, both
this change of frame and its inverse are bounded with respect to \(1\oplus g\) near \(D\),
uniformly. By \eqref{eq:tangentmodel} and the multiweights of the adapted frame
\(1,\partial_{z_1},\dots,\partial_{z_n}\), \(h\) is therefore uniformly adapted to
\(E_*\). Theorem~\ref{thm:HEcriterion} gives parabolic polystability.\qed

\section{Proof of Theorem C}\label{sec:thmC}

\subsection{Parabolic Higgs bundles and the Higgs degree identity}
The \emph{logarithmic cotangent bundle} is the dual \(\OmLog:=\Tlog^*\) of the
logarithmic tangent bundle; in adapted coordinates \eqref{eq:adap-coor} it is generated by
the coframe
\[
\frac{dz_1}{z_1},\dots,\frac{dz_\ell}{z_\ell},
dz_{\ell+1},\dots,dz_n
\]
dual to \eqref{eq:logtangentframe}. We use the standard logarithmic definition of a
parabolic Higgs bundle; cf.\ \cite[\S3.1]{Mochizuki}.

\begin{definition}\label{def:parhiggs}
A \emph{parabolic Higgs bundle} on \((X,D)\) is a parabolic bundle \(E_*\)
together with \(\theta\in H^0\bigl(X,\End(E)\otimes\OmLog\bigr)\) such that
\(\theta\wedge\theta=0\) and \(\theta\) preserves every parabolic filtration, that is,
\(\theta\bigl(F^i_c(E_*)\bigr)\subset F^i_c(E_*)\otimes\OmLog|_{D_i}\) for all \(i,c\).
It is \emph{semistable} if \(\mu_\Omega(\cV_*)\le\mu_\Omega(E_*)\) for every saturated
\(\theta\)-invariant subsheaf \(\cV\subset E\) of positive proper rank, \emph{stable} if
the inequality is always strict, and \emph{polystable} if it is a direct sum of stable
parabolic Higgs bundles of the same parabolic slope, the decomposition splitting every
filtration and making \(\theta\) block diagonal.
\end{definition}

For a Hermitian metric \(h\) on \(E|_{\Xo}\) and a K\"ahler form \(\omega\) on \(\Xo\),
the Hitchin--Simpson equation is
\begin{equation}\label{eq:HS}
   \mathrm{i}\Lambda_\omega\bigl(F_h+[\theta,\theta^{*h}]\bigr)=\mu\Id_E .
\end{equation}
The next statement is the parabolic cone-metric version of the Chern--Weil argument for
Higgs bundles of \cite[Lemma~3.2 and Prop.~3.3]{Simpson}; it extends Theorems~\ref{thm:HEcriterion} and \ref{thm:degree} by introducing a Higgs field.

\begin{proposition}\label{prop:higgsdegree}
Let \((E_*,\theta)\) be a parabolic Higgs bundle on \((X,D)\) whose underlying parabolic
bundle is locally abelian and satisfies \eqref{eq:oneweight}, let \(\omega\) be a cone
metric with \(\Omega=[\omega]\), and let \(h\) be uniformly adapted to \(E_*\), with
\(|\theta|_{h,\omega}\) bounded and \eqref{eq:HS} satisfied. Then for every saturated
\(\theta\)-invariant \(\cV\subset E\) of rank \(k\),
\begin{equation}\label{eq:higgsdegree}
   \pardeg_\Omega(\cV_*)
   =\frac{\mu k}{2\pi}\Vol_\omega(\Xo)
    -\frac1{2\pi}\int_{\Xo}\bigl(|\bar\partial\pi_\cV|^2+|[\theta,\pi_\cV]|^2\bigr)dV_\omega,
\end{equation}
both integrals being finite, and \((E_*,\theta)\) is parabolic Higgs polystable of slope
\(\mu\Vol_\omega(\Xo)/2\pi\).
\end{proposition}

\begin{proof}
Let \(Z\) be an exceptional set supplied by Lemma~\ref{lem:goodlocus}. On
\(\Xo\setminus Z\) write \(E=\cV\oplus\cV^{\perp}\). Since \(\cV\) is
\(\theta\)-invariant,
\(\theta=\begin{psmallmatrix}\theta_1&\varrho\\0&\theta_2\end{psmallmatrix}\) and
\([\theta,\pi_\cV]=-\varrho\). The \(\cV\)-block of \([\theta,\theta^{*h}]\) is
\([\theta_1,\theta_1^{*}]+\varrho\wedge\varrho^{*}\); the contracted trace of a commutator
vanishes and \(\tr\mathrm{i}\Lambda_\omega(\varrho\wedge\varrho^{*})=|\varrho|^2\), so
\eqref{eq:HS} gives
\begin{equation}\label{eq:higgs-trace}
   \tr\bigl(\pi_\cV\,\mathrm{i}\Lambda_\omega F_h\bigr)=\mu k-|[\theta,\pi_\cV]|^2
   \qquad\text{on }\Xo\setminus Z .
\end{equation}
Substituting \eqref{eq:higgs-trace} into \eqref{eq:CWpointwise}, the proof of
Theorem~\ref{thm:degree} applies verbatim: the
determinant estimates of Lemmas~\ref{lem:upperdet} and~\ref{lem:lowerdet}, the
integrability of Lemma~\ref{lem:weightedL1} and the cut-off of Lemma~\ref{lem:cutoff}
involve only \(h\), \(\omega\) and the parabolic structure, while
\(|\Lambda_\omega F_h|_h\le C\) by \eqref{eq:HS} and the boundedness of \(\theta\).
This gives \eqref{eq:higgsdegree}, with the nonnegative integrand
\(|\bar\partial\pi_\cV|^2+|[\theta,\pi_\cV]|^2\) in place of \(|\bar\partial\pi_\cV|^2\),
and both integrals are finite.

Taking \(\cV=E\) gives \(\mu_\Omega(E_*)=\mu\Vol_\omega(\Xo)/2\pi\), so
\eqref{eq:higgsdegree} is semistability, with equality if and only if
\(\bar\partial\pi_\cV=0\) and \([\theta,\pi_\cV]=0\); if no such \(\cV\) attains equality,
\((E_*,\theta)\) is stable. Otherwise, as in the
proof of Theorem~\ref{thm:HEcriterion}, \(\pi_\cV\) extends to an idempotent
\(\pi\in P(E_*)\) with \(\im\pi=\cV\), and Lemma~\ref{lem:directsummands} gives
\(E_*=E_{1,*}\oplus E_{2,*}\) splitting every filtration. The section
\([\theta,\pi]\in H^0\bigl(X,\End(E)\otimes\OmLog\bigr)\) vanishes on the dense open set
\(\Xo\setminus Z\), hence vanishes identically, so \(\theta\)
is block diagonal for this decomposition and each summand is a parabolic Higgs bundle. The
restricted metrics are uniformly adapted to the corresponding parabolic structures and
satisfy \eqref{eq:HS} with the same \(\mu\), by orthogonality of the splitting, and
\(|\theta_s|_{h_s,\omega}\le|\theta|_{h,\omega}\). The weight condition
\eqref{eq:oneweight} is inherited by the summands, so induction on the rank completes the
proof.
\end{proof}

\subsection{Simpson's canonical parabolic Higgs bundle}\label{ss:canonical}

\begin{definition}\label{def:HSbundle}
Let \(E:=\cO_X\oplus TX\) with the direct sum parabolic structure
\begin{equation}\label{eq:thetaHS}
F^i_c(E_*):=\cO_{D_i}\oplus F^i_c(TX_*) , \qquad
   \theta:=\begin{pmatrix}0&0\\ \Id_{TX}&0\end{pmatrix} .
\end{equation}
Here
\[
\Id_{TX}\in H^0(X,TX\otimes\Omega^1_X)\subset H^0(X,TX\otimes\OmLog) .
\]
\end{definition}

In the smooth case \((E,\theta)\) is the dual of Simpson's canonical system of Hodge
bundles for ball uniformization \cite[Props.~9.8--9.9]{Simpson}. We call \(\theta\) the
\emph{tautological Higgs field}: evaluated on a tangent vector it is
\(\theta(v)(a,u)=(0,av)\), where \(v\in TX\) and \((a,u)\in\cO_X\oplus TX\); so \(\theta\) is holomorphic across \(D\), although it is used as a logarithmic Higgs
field. By the direct-sum definition of the filtrations and
\eqref{eq:tangentmultiweights}, \(E_*\) is locally abelian: in an adapted chart
\eqref{eq:adap-coor}, \(1,\partial_{z_1},\dots,\partial_{z_n}\) is an adapted frame, the
first vector of multiweight zero and the remaining ones with the multiweights
\eqref{eq:tangentmultiweights}; in particular \(\Gr^i_{w_i}(E_*)=N_i\), and
\eqref{eq:TXzero} gives \((E)_{\mathbf 0}=\cO_X\oplus\Tlog\).

The block form \eqref{eq:thetaHS} gives \(\theta\wedge\theta=0\). In adapted coordinates
\eqref{eq:adap-coor},
\begin{equation}\label{eq:idlog}
   \Id_{TX}=\sum_{j=1}^{\ell}(z_j\partial_{z_j})\otimes\frac{dz_j}{z_j}
            +\sum_{j=\ell+1}^{n}\partial_{z_j}\otimes dz_j ,
\end{equation}
so that, in the logarithmic coframe,
\[
   \theta=\sum_{j=1}^{n}A_j\otimes\varepsilon_j,
   \qquad
   \varepsilon_j:=\begin{cases}dz_j/z_j,&j\le\ell,\\ dz_j,&j>\ell,\end{cases}
   \qquad
   A_j(a,u):=\begin{cases}(0,a\,z_j\partial_{z_j}),&j\le\ell,\\
                          (0,a\,\partial_{z_j}),&j>\ell .\end{cases}
\]
Fix \(j\le\ell\). By \eqref{eq:idlog}, the residue \(\Res_{D_j}\theta=A_j|_{D_j}\)
vanishes, because \(z_j\partial_{z_j}\) does; every remaining coefficient restricts on
\(D_j\) to a section of \(TD_j\); and every \(A_m\) annihilates the \(TX\) summand. Hence,
for \(c<w_j\), the step \(F^j_c(E_*)=\cO_{D_j}\oplus TD_j\) is preserved by every
\(A_m|_{D_j}\), while for \(c\ge w_j\) the step is the whole bundle and preservation is
automatic. Thus \((E_*,\theta)\) is a parabolic Higgs bundle.

\subsection{Proof of Theorem C}
Let \(\omega\) be a K\"ahler--Einstein cone metric as in Theorem~C, with
\(\Ric(\omega)=\lambda\omega\), \(\lambda<0\), and let \(g\) be the induced Hermitian
metric on \(TX|_{\Xo}\). In the proof of Theorem~B the diagonal blocks of
\(\mathrm{i}\Lambda_\omega F_{h_\omega}\) are \(t^2n\) and \(\lambda-t^2\), which agree
only for \(t^2=\lambda/(n+1)\), impossible when \(\lambda<0\). For the direct sum
\(\cO_X\oplus TX\) with the metric \(1\oplus g\), the Higgs commutator in \eqref{eq:HS}
contributes with the opposite signs. Set
\begin{equation}\label{eq:tnegative}
   t:=\sqrt{\frac{-\lambda}{n+1}}.
\end{equation}

Let \(h:=1\oplus g\) on \(E|_{\Xo}\). By \eqref{eq:tangentmodel} and the multiweights of
the adapted frame \(1,\partial_{z_1},\dots,\partial_{z_n}\), \(h\) is uniformly adapted to
\(E_*\), and \(|t\theta|^2_{h,\omega}=t^2n\) is constant. In a \(g\)-unitary coframe,
\[
   \mathrm{i}\Lambda_\omega\bigl[t\theta,(t\theta)^{*h}\bigr]
   =\begin{pmatrix}-t^2n&0\\0&t^2\Id_{TX}\end{pmatrix};
\]
its trace vanishes, as it must for a commutator.
Since \(F_h=0\oplus F_g\) and \(\mathrm{i}\Lambda_\omega F_g=\lambda\Id_{TX}\) by
Section~\ref{sec:thmA}, the left side of
\eqref{eq:HS} equals \(-t^2n\) on \(\cO_X\) and \((\lambda+t^2)\Id\) on \(TX\).
By \eqref{eq:tnegative}, these agree and have common value \(n\lambda/(n+1)\).
Proposition~\ref{prop:higgsdegree} shows that \((E_*,t\theta)\) is polystable.
Since \(t>0\), the \(\theta\)-invariant parabolic subsheaves and decompositions are
exactly the \(t\theta\)-invariant ones. Thus \((E_*,\theta)\) is parabolic Higgs
polystable.\qed

\section{Proof of Theorem D}\label{sec:thmD}

\subsection{Parabolic Chern classes}
Throughout this section \(n\ge2\). For \(i<j\), set \(D_{ij}:=D_i\cap D_j\). If \(E_*\) is
locally abelian, its common adapted frames make
\begin{equation}\label{eq:simultaneousgraded}
   \Gr^{i,j}_{a,b}(E_*):=
   \frac{F^i_a|_{D_{ij}}\cap F^j_b|_{D_{ij}}}
   {\bigl(F^i_{<a}|_{D_{ij}}\cap F^j_b|_{D_{ij}}\bigr)
    +\bigl(F^i_a|_{D_{ij}}\cap F^j_{<b}|_{D_{ij}}\bigr)}
\end{equation}
a vector bundle on each connected component \(P\subset D_{ij}\); denote its rank there
by \(\rk_P\Gr^{i,j}_{a,b}(E_*)\).

\begin{definition}\label{def:parchern}
For a locally abelian parabolic bundle \(E_*\), Mochizuki's formula in our convention
\cite[\S\S3.1.2, 3.1.5]{Mochizuki} is
\begin{align}
   \parch_2(E_*)={}&\ch_2(E)
   -\sum_i\sum_{a\in\operatorname{Wt}_i(E_*)}
      a\,(\iota_i)_*c_1\bigl(\Gr^i_a(E_*)\bigr) \notag\\
   &+\frac12\sum_i\sum_{a\in\operatorname{Wt}_i(E_*)}
      a^2\rk\Gr^i_a(E_*)[D_i]^2 \notag\\
   &+\sum_{i<j}\sum_{P\in\pi_0(D_{ij})}
      \left(\sum_{\substack{a\in\operatorname{Wt}_i(E_*)\\
                             b\in\operatorname{Wt}_j(E_*)}}
      ab\,\rk_P\Gr^{i,j}_{a,b}(E_*)\right)[P].
   \label{eq:Mochizuki-pch2}
\end{align}
Here \([P]\in H^4(X,\R)\) is the class of \(P\). We set
\begin{equation}\label{eq:parc2def}
   \parc_2(E_*):=\frac12\parc_1(E_*)^2-\parch_2(E_*).
\end{equation}
These are also the formulas used in \cite[Defs.~4.7 and 4.17]{dBP}.
\end{definition}

Write \(\ch_2(X):=\ch_2(TX)\).

\begin{proposition}
\label{prop:rankclasses}
Let \(E_*\) be any of the parabolic bundles in Theorems A, B, C. Then
\begin{align}
   \parc_1(E_*)&=c_1(X,\Delta),\label{eq:pc1TX}\\
   \parch_2(E_*)&=\ch_2(X)-\sum_i\Bigl(w_i-\frac{w_i^2}{2}\Bigr)[D_i]^2,
      \label{eq:pch2TX}\\
   \parc_2(E_*)&=c_2(X,\Delta).\label{eq:pc2TX}
\end{align}
\end{proposition}

\begin{proof}
The exact sequence \eqref{eq:Tianextension} and the direct-sum description of \(E\)
show that, in all three cases,
\[
   c_1(E)=c_1(X),\qquad \ch_2(E)=\ch_2(X).
\]
Along \(D_i\), the only positive-weight graded piece is
\(\Gr^i_{w_i}(E_*)=N_i\), of rank one. Moreover, the adapted frames
\(\partial_{z_1},\dots,\partial_{z_n}\) and \(1,\partial_{z_1},\dots,\partial_{z_n}\) of
Sections~\ref{sec:thmA}--\ref{sec:thmC} have the multiweights
\eqref{eq:tangentmultiweights}, so no adapted frame vector has positive weight along two
distinct components; hence all simultaneous graded
pieces \eqref{eq:simultaneousgraded} with \(a,b>0\) vanish. Mochizuki's formulas
\eqref{eq:parc1} and \eqref{eq:Mochizuki-pch2} therefore give
\[
   \parc_1(E_*)=c_1(X)-\sum_iw_i[D_i]=c_1(X,\Delta)
\]
by \eqref{eq:logc1}, and
\[
   \parch_2(E_*)
   =\ch_2(X)-\sum_iw_i(\iota_i)_*c_1(N_i)
      +\frac12\sum_iw_i^2[D_i]^2.
\]
Since \(N_i\cong\cO_X(D_i)|_{D_i}\), one has
\((\iota_i)_*c_1(N_i)=[D_i]^2\), proving \eqref{eq:pch2TX}.

To compute \(\parc_2\), expand \eqref{eq:parc2def} and use
\(\ch_2(X)=\frac12c_1(X)^2-c_2(X)\):
\[
   \parc_2(E_*)=c_2(X)-\sum_iw_i\bigl(c_1(X)-[D_i]\bigr)[D_i]
     +\sum_{i<j}w_iw_j[D_i][D_j].
\]
Adjunction gives
\((\iota_i)_*c_1(TD_i)=\bigl(c_1(X)-[D_i]\bigr)[D_i]\), so the right-hand side equals
\(c_2(X,\Delta)\) by \eqref{eq:logc2}.
\end{proof}

\subsection{The Bogomolov--Gieseker inequality and the end of the proof}
Mochizuki \cite{Mochizuki} proved the parabolic Bogomolov--Gieseker inequality for parabolic Higgs bundles
in the projective setting, and Jiang--Li \cite{JiangLiBG} recently extended it to compact K\"ahler manifolds.
The following is the locally abelian bundle case of \cite[Theorem~A]{JiangLiBG}, written in our convention.

\begin{theorem}[{Jiang--Li \cite[Theorem~A]{JiangLiBG}}]\label{thm:JiangLiBG}
Let \(X\) be a compact K\"ahler manifold of dimension \(n\ge2\), let \(D\subset X\) be a
simple normal crossing divisor, and let \((E_*,\theta)\) be a locally abelian parabolic
Higgs bundle of rank \(r\), stable with respect to a K\"ahler class \(\Omega\). Then
\begin{equation}\label{eq:BGstable}
   \parch_2(E_*)\cdot\Omega^{n-2}
   \le\frac{\parc_1(E_*)^2\cdot\Omega^{n-2}}{2r}.
\end{equation}
\end{theorem}

We will use the following immediate extension to the polystable case.

\begin{lemma}\label{lem:BGpoly}
Inequality \eqref{eq:BGstable} also holds for \(\Omega\)-polystable locally abelian
parabolic Higgs bundles.
\end{lemma}

\begin{proof}
Write \((\mathcal G_*,\theta)=\bigoplus_\alpha(\mathcal G_{\alpha,*},\theta_\alpha)\) with
stable factors of the same parabolic slope. Since the decomposition splits every
parabolic filtration, Lemma~\ref{lem:directsummands} shows that each
\(\mathcal G_{\alpha,*}\) is locally abelian. Put \(r_\alpha:=\rk\mathcal G_\alpha\),
\(r:=\sum_\alpha r_\alpha\), \(\gamma_\alpha:=\parc_1(\mathcal G_{\alpha,*})\) and
\(\gamma:=\sum_\alpha\gamma_\alpha\). Since the parabolic Chern character is additive on
direct sums,
\[
   \parc_1(\mathcal G_*)=\gamma,
   \qquad
   \parch_2(\mathcal G_*)=\sum_\alpha\parch_2(\mathcal G_{\alpha,*}),
\]
and summing \eqref{eq:BGstable} over the factors,
\begin{equation}\label{eq:BGsum}
   \parch_2(\mathcal G_*)\cdot\Omega^{n-2}
   \le\sum_\alpha\frac{\gamma_\alpha^2\cdot\Omega^{n-2}}{2r_\alpha}.
\end{equation}
Put \(\zeta_\alpha:=\gamma_\alpha-(r_\alpha/r)\gamma\). A direct expansion gives
\begin{equation}\label{eq:BGdecomp}
   \sum_\alpha\frac{\gamma_\alpha^2}{2r_\alpha}
   =\frac{\gamma^2}{2r}+\sum_\alpha\frac{\zeta_\alpha^2}{2r_\alpha}.
\end{equation}
Equality of slopes means \(\gamma_\alpha\cdot\Omega^{n-1}/r_\alpha=\gamma\cdot\Omega^{n-1}/r\),
so each \(\zeta_\alpha\) is a real \((1,1)\)-class with \(\zeta_\alpha\cdot\Omega^{n-1}=0\),
that is, primitive; the Hodge--Riemann bilinear relations give
\(\zeta_\alpha^2\cdot\Omega^{n-2}\le0\).
Combining \eqref{eq:BGsum} and \eqref{eq:BGdecomp} with the Hodge--Riemann inequality
gives \eqref{eq:BGstable} for \(\mathcal G_*\).
\end{proof}

\begin{proof}[Proof of Theorem~D]
Write \(\Ric(\omega)=\lambda\omega\). Taking cohomology classes in
\eqref{eq:Ric-current} gives
\begin{equation}\label{eq:KEclass}
   c_1(X,\Delta)=\frac{\lambda}{2\pi}\,\Omega .
\end{equation}
If \(\lambda=0\), then
\[
   \parc_1(TX_*)=c_1(X,\Delta)=0
\]
by \eqref{eq:KEclass} and \eqref{eq:pc1TX}. By Theorem~A, \(TX_*\) is parabolic
polystable. Hence Lemma~\ref{lem:BGpoly} gives
\[
   \parch_2(TX_*)\cdot\Omega^{n-2}\le0 .
\]
Therefore
\[
   \parc_2(TX_*)\cdot\Omega^{n-2}\ge0 ,
\]
which is \eqref{eq:MYintro} in this case by \eqref{eq:pc1TX} and \eqref{eq:pc2TX}.

Suppose \(\lambda\ne0\), and let \((E_*,\theta)\) be the parabolic Tian extension of
Definition~\ref{def:Tian-parabolic} with \(\theta=0\) if \(\lambda>0\), and the canonical
parabolic Higgs bundle of Definition~\ref{def:HSbundle} if \(\lambda<0\). Then \(E_*\) has
rank \(n+1\) and \((E_*,\theta)\) is parabolic Higgs polystable with respect to
\(\Omega\), by Theorems~B and~C. Hence Lemma~\ref{lem:BGpoly} gives
\[
   \parch_2(E_*)\cdot\Omega^{n-2}\le\frac{\parc_1(E_*)^2\cdot\Omega^{n-2}}{2(n+1)} .
\]
Substituting \(\parch_2=\frac12\parc_1^2-\parc_2\) and rearranging,
\[
   \Bigl(2(n+1)\parc_2(E_*)-n\,\parc_1(E_*)^2\Bigr)\cdot\Omega^{n-2}\ge0,
\]
and \eqref{eq:pc1TX} and \eqref{eq:pc2TX} identify \(\parc_1(E_*)=c_1(X,\Delta)\) and
\(\parc_2(E_*)=c_2(X,\Delta)\), giving \eqref{eq:MYintro}.
\end{proof}

\bibliographystyle{amsplain}
\bibliography{ref}

\end{document}